\documentclass{article}
\usepackage{amsmath}
\usepackage{amssymb}
\usepackage{amsthm}
\usepackage{mathrsfs}
\usepackage{enumerate}
\usepackage{color}
\usepackage{geometry}
\usepackage{bbm}
\usepackage{slashed}
\usepackage{graphicx}
\usepackage{hyperref} 
\usepackage{enumerate,cancel,amssymb,centernot}
\usepackage{latexsym}
\usepackage{graphicx}
\usepackage{tikz}
\usepackage{tikz-network}
\usepackage{chngcntr}
\usepackage{amsmath}
\usepackage[caption=false,farskip=0pt,labelfont={bf}]{subfig}
\allowdisplaybreaks[4]
\usetikzlibrary{shapes,arrows}

\newcounter{cnstcnt}

\newtheorem{thm}{Theorem}[section]
\newtheorem{lem}[thm]{Lemma}

\newtheorem{rmk}[thm]{Remark}
\newtheorem{prop}[thm]{Proposition}
\newtheorem{defi}[thm]{Definition}
\numberwithin{equation}{section}

\def\eps{\epsilon}

\def\P{\mathbb{P}}
\def\cA{\mathcal{A}}

\newcommand{\cC}{\mathcal{C}}

\newcommand{\sB}{\mathcal{B}}

\def\0{\textbf{0}}

\def\Tri {\mathrm{STri}}
\def\Dtri {\mathrm{DTri}}

\def\deg {\mathrm{deg}}
\def\dist{\mathrm{dist}}
\def\Prr {\mathrm{Prr}}
\def\Noi {\mathrm{Noi}}

\def\sB {\mathscr B}

\def\sC {\mathscr C}

\def\fS {\mathsf{S}}
\def\U {\mathsf{U}}
\def\frB{\mathfrak{B}}

\DeclareSymbolFont{sfoperators}{OT1}{ptm}{m}{n}

\makeatletter
\usepackage{graphicx} 

\title{Non-ergodicity for the noisy majority vote process on trees}
\author{Jian Ding\\Peking University\and Fenglin Huang \\Peking University}
\date{}

\begin{document}
\maketitle
\begin{abstract}
    We consider the noisy majority vote process on infinite regular trees with degree $d\geq 3$, and we prove the non-ergodicity, i.e., there exist multiple equilibrium measures.  Our work extends a result of Bramson and Gray (2021) for $d\geq 5$.
\end{abstract}
\section{Introduction}
\subsection{Noisy majority vote process}
We consider the noisy majority vote process  on a connected graph $G$ (with vertex set $V$). For each vertex $x \in V,$ let $N(x)$ be the collection of neighbors of $x$. In this paper, we focus on the case that $G = \mathbb T_d$ where $\mathbb T_d$ is the infinite tree in which each site has $d$ neighbors.

The noisy majority vote process with noise $\eps\in (0,1)$ is defined as follows. At each discrete time $t$, each vertex possesses a spin $-1$ or $1$, denoted as $\sigma(t)\in \{-1,1\}^V$. At each discrete time $t$, one of the two update events can occur at every vertex $x$. The first is a \textbf{vote event}, in which the spin at $x$ is updated to the majority of the spins in $N(x)$ at time $t-1$ (where we regard $N(x)$ as the voting neighborhood for $x$). When a tie vote occurs, the spin at $x$ does not change, i.e. $\sigma_x(t)=\sigma_x(t-1)$. Vote events are assumed to occur independently with probability $1 -\eps$ at each site. The other type of event is a \textbf{noise event}, which occurs with probability $\eps$ at each site, where the spin of the chosen vertex is updated to -1 and 1 with equal probability. Put in the context of social sciences, this process may describe the evolution of opinions of a population when the underlying graph is the social network (where an opinion is identified with a spin).

The main purpose of the paper is to consider whether the dynamic is ergodic (i.e., whether the equilibrium measure is unique). In \cite{BG21}, they have shown that when $d\ge 5$, there are uncountably many mutually singular equilibria for small enough $\eps>0$; thus, the dynamic is not ergodic. When $d=2$, since $\mathbb T_2= \mathbb Z$, a general result in \cite{gray1982positive} implies that the noisy majority vote process has a unique equilibrium for all $\eps > 0 $.
\begin{thm}\label{thm-main}
    For any $d\ge 3$, the noisy majority vote process on $\mathbb T_d$ is not ergodic for small enough $\eps>0$.
\end{thm}

\subsection{Related works}

Our work falls under the umbrella of how disorder/noise affects fundamental properties of physical systems, for which the spin system is a particularly important class of examples.  There are a number of ways to introduce disorder to a spin system as we elaborate next.

\begin{enumerate}[i]
    \item Spin models with random external fields.

A prominent example is the random field Ising model. By introducing random external fields to spins, the RFIM studies how random perturbations compete with the spin interactions, leading to complex and rich behaviors that have motivated extensive research. Considerable progress has been made on the RFIM, including phase transitions in three and higher dimensions (\cite{IM75, Imr85, BK88, DZ24, DLX24}) and quantitative decays on magnetization in two dimensions (\cite{AW89, AW90, Cha18, AP19, DX21, AHP20, DW23,  ding2023phase}). In more generality, spin systems with random external fields, including the random field Potts model and the random field XY model, were studied (\cite{DZ24,DHP24,CR24}).    
\item Spin models with disordered interactions.

As a mean-field model, the Sherrington-Kirkpatric spin glass model has been extensively studied, and many important results have been established; see e.g., \cite{Pan13, Tal11I, Tal11II} for a not-up-to-date summary of existing results. On the contrary, the model on lattices, known as the Edwards-Anderson spin glass model, is much more challenging: despite progress made in \cite{chatterjee2023spin, AD14, ADNS10, AH20, ANS19, ANS21, ANSW14, ANSW16, NS98, NS92, NS01, NS03}, many important and basic properties remain elusive.
\item Non-equilibrium models, such as the non-equilibrium random field Ising model. 

In a recent paper \cite{ding2024dynamical}, the authors investigated both ground state evolution and Glauber dynamics of the zero-temperature random field Ising model. For the ground state evolution, they proved that there is no avalanche in two dimensions, and there is a phase transition for avalanche in three dimensions and higher as the strength of the disorder varies. The avalanche for the Glauber dynamics seems of major challenge, and the authors can only prove some preliminary results that most spin flips occur around some critical threshold as the noise strength vanishes.
\item  Noisy dynamics in spin systems, such as the noisy majority vote process. 

The well-known positive rate conjecture posits that any one-dimensional, finite-range spin model with a positive flipping rate is ergodic. In \cite{gray1982positive}, the conjecture for nearest-neighbor and attractive models was proved. However, a counterexample provided in \cite{Gac86, Gac01} disproved the conjecture. For higher dimensions, progress has been limited. Ergodicity has been established for two-dimensional oriented lattices \cite{Too80}, while for infinite regular trees, non-ergodicity has been demonstrated for $d\ge 4$ in the oriented case and $d\ge 5$ in the unoriented case \cite{BG21}.
\end{enumerate}

\subsection{A word on proof strategy for $\mathbb{T}_3$}
In \cite{BG21}, the authors constructed a family of uncountably many initial configurations with different limiting distributions. The key idea in their proof is to introduce an auxiliary process that dominates the deviation of the original process from its initial configuration. This domination allows them to derive how interfaces between plus and minus spins evolve from the initial configuration. However, in the unoriented case, extending this domination to $d=3,4$ seems challenging, which prevents them from proving the non-ergodicity for $\mathbb T_3$ and $\mathbb T_4$.

In this paper, we adopt a different strategy. The main novelty lies in the case $d=3$ and the extension to $d\ge 3$ follows naturally from the framework developed for $d=3$. In order to establish non-ergodicity, we show that for the initial configuration of all plus, the probability of a vertex being $1$ has a limit strictly bigger than $\frac{1}{2}$, as $t\to \infty$. Our intuition is as follows: if a set of vertices $A$ is voted to become -1 at time $t$, there must exist at least $|A|+1$ minus vertices at time $t-1$ that are responsible for these votes. Using this insight, we apply induction on $t$ to obtain an upper bound on the probability that a set of vertices becomes -1, as detailed in Section~\ref{sec: induction procedure}.

To carry out the induction procedure, a key step is to control the enumeration of possible configurations at time $t-1$ that can lead to the target configuration at time $t$. To this end, we prove two combinatorial results that upper-bound the enumeration of such configurations, as detailed in Lemmas \ref{lem: enumeration induced relation} and \ref{lem: enumeration lemma final}.

\subsection{The extension to $d\ge 3$}
In order to extend to the case $d\ge 3$, we introduce the following model. From now on, we fix the initial configuration as all plus, i.e. $\sigma_x(0)=1,~~\forall x\in V$.

For any $d\ge 3$, we consider a variation of the majority vote model, which will be referred to as the {\bf minus biased vote process}.
The minus biased vote process with noise $\eps\in (0,1)$ is defined as follows.  At each discrete time, each vertex possesses a spin $-1$ or $1$, denoted as $\sigma(t)\in \{-1,1\}^V$.  Recall that we have assumed that the initial configuration is all plus, and at each time $t \in \mathbb N$ one of the following two updates occurs at each vertex $x\in V$ independently: \begin{enumerate}[(1)]
    \item  with probability $1-\epsilon$, a {\bf minus bias update} occurs which updates the spin at $x$ to $-1$ if $|\{y\in N(x): \sigma_y(t-1) = -1\}| \geq 2$ and updates the spin at $x$ to $1$ otherwise; 
    \item with probability $\epsilon$, a {\bf noise update} occurs which updates the spin at $x$ to $-1$ and $1$ with probability half each.
\end{enumerate}
We use the notation $\P$ to denote the probability with respect to the minus biased vote process.

\begin{thm}\label{thm: minus main higher dimension}
    For any $d\ge 3$, we have \begin{equation}\label{eq: minus main higher dimension}
        \P(\sigma_{v}(t)=-1)\le c<0.5,~~~~\mbox{for all }v\in \mathbb{T}_d 
    \end{equation} for some constant $c>0$ only depending on $d$.
\end{thm}
Since the noise majority vote process started at all plus dominates the minus biased vote process, Theorem~\ref{thm-main} follows from Theorem~\ref{thm: minus main higher dimension} directly. The next theorem considers general infinite 
trees.
\begin{thm}\label{thm: minus main general tree}
    For any $d\ge 3$ and any infinite tree $\mathbb{T}$ with maximal degree $d$, we have \begin{equation}\label{eq: minus main general tree}
        \P(\sigma_{v}(t)=-1)\le c<0.5,~~~~\mbox{for all }v\in \mathbb{T} 
    \end{equation} for some constant $c>0$ only depending on $d$.
\end{thm}
Since the minus biased vote process on $\mathbb{T}$ dominates the minus biased vote process on $\mathbb{T}_{d}$, we get that Theorem~\ref{thm: minus main general tree} follows from Theorem~\ref{thm: minus main higher dimension} directly.
\section{Proof of Theorem \ref{thm: minus main higher dimension}: the induction procedure}\label{sec: induction procedure}
In this subsection, we will outline the proof of Theorem \ref{thm: minus main higher dimension}, while postponing a number of lemmas for smooth flow of presentation. 
\subsection{Notations}
In this section, we describe the notations used in the following sections.  We use $C, C_i$ to represent positive constants whose actual values may vary from line to line. In addition, we use $c,c_i$ for constants whose values are fixed (throughout the paper) upon their first occurrences.
\begin{defi}
    For any positive integer $k$, we use the notation $[k]$ to denote the set $\{1,\cdots,k\}$.
\end{defi}
\begin{defi}
     Let $\sigma\in \{-1,1\}^{V}$ be a configuration on $V$ and let $A\subset V$ be a vertex set. We denote by $\sigma_A$ the configuration restricted on $A$. Let $\deg_{G}(v)$ be the degree of $v$ in the graph $G$. For any $V'\subset V$, let $H=G|_{V'}$ be the subgraph of $G$ restricted on $V'$.
     For $v\in H$, we call $v$ a \textbf{branching vertex} in $H$ if $\deg_{H}(v)\ge 3$, and we denote by $V(H)$ the vertex set of $H$.
\end{defi}
\begin{defi}
   Recall that $N(x)$ is the collection of all the neighbors of $x$. For a vertex set $A$ we denote $N(A) = \cup_{x\in A} N(x)$ the neighbors of $A$.
    
    For a vertex $x$, we say $y$ is its double neighbor if $\dist(x,y)=2$ (where $\dist$ denotes the graph distance). Specifically, if $x$ is an odd vertex and $\dist(x,y)=2$, then we say $y$ is an odd neighbor of $x$; if $x$ is an even vertex and $\dist(x,y)=2$, then we say $y$ is an even neighbor of $x$.
\end{defi}
\begin{defi}
    For a vertex set $A=\{v_1,v_2,\cdots,v_n\}$ and for $p\geq 3$, we say $x$ is a type $p$ \textbf{single trifurcation} of $A$ if $|N(x)\cap A|=p$. In addition, we say $x$ is a type $p$ \textbf{double trifurcation} of $A$ if there exists  $Y\subset N(x)$ with $|Y|=p$ such that  $N(y)\cap (A\setminus\{x\})\neq\emptyset$ for all $y\in Y$. Let $\Tri(A,p)$ denote the number of type $p$ single  trifurcations and let $\Dtri(A,p)$ denote the number of type $p$ double trifurcations. In addition, let $\Tri(A)=\sum_{p=3}^d(p-2)\Tri(A,p)$ and let $\Dtri(A)=\sum_{p=3}^d(p-2)\Dtri(A,p)$. We say $x$ is a single (resp. double) trifurcation if it is a type $p$ single (resp. double) trifurcation for some $p\geq 3$. See Figure~\ref{fig: trifurcations} for an illustration.
\end{defi}
\begin{figure}[htbp]
    \centering
    \captionsetup[subfigure]{width=4cm}
    \subfloat[A single trifurcation]{\includegraphics[width=1.8cm,height=3.9cm]{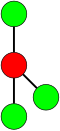}}
    \hspace{5cm}
    \subfloat[A double trifurcation]{\includegraphics[width=2.35cm,height=3.65cm]{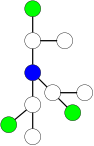}}
    \caption{Illustrations of single and double trifurcations. Red point: a type 3 single trifurcation. Blue point: a type 3 double trifurcation. Green points: vertices in $A$.}
    \label{fig: trifurcations}
\end{figure}

\begin{defi}\label{def: nearby}
    For a collection of odd points $A$ (i.e., an odd vertex set), we say $A$ is an \textbf{odd cluster} if it is connected with respect to the double-neighboring relation. We define the size of an odd cluster, denoted by $|\cdot|$, to be the number of odd vertices in the cluster.
    
    For two vertex sets $A,B\subset V$, let $\dist(A,B)=\min_{x\in A,y\in B}\dist(x,y)$ denote the distance between $A$ and $B$. For two disjoint odd vertex sets $A$ and $B$, we say they are \textbf{adjacent} if $\dist(A, B) = 2$. 
    The same definitions also apply to even vertex sets.
\end{defi}
\begin{rmk}
    We remark that an odd cluster is not necessarily an odd connected component.
\end{rmk}
Throughout this paper, depth-first searches on trees will always start from leaves. In addition, a depth-first search on a forest will be a concatenation of depth-first searches on the sub-trees of the forest.
\begin{defi}\label{def: far neighbor}
Let $H$ be a tree with maximal degree $d$ and fix a depth-first search of $H$. We can get an order for vertices in $H$ according to their first appearance in the depth-first search.
    For any branching vertex $x$ in $H$, let $l$ be the order of $x$. Let $x_-,x_+$ be the vertex with order $l-1$ and $l+1$ respectively, and we call the other neighbors of $x$ the far neighbors of $x$. 
\end{defi}

\subsection{Induction hypothesis via update history}\label{sec: induction hyporthesis}
For any vertex $v\in V$, if $\sigma_v(t)=-1$, then there are two possibilities depending on whether a vote event or a noise event occurs at $v$ at time $t$: in the case of a vote update we say $v$ is a \textbf{real minus} at time $t$, and in the case of a noise event we say $v$ is a \textbf{noise minus} at time $t$. By the definition of the vote event, for any real minus $v$ at time $t$, there exist at least two minuses (i.e., vertices with minus spin values) in $N(v)$ at time $t-1$. In order to compute the probability that $\sigma_v(t)=-1$, we need to trace the update history backward from time $t$ and iteratively record those minuses in the previous time that are responsible for producing the real minuses at the current time.

The update history for a vertex $v$ backward from time $T$ can be represented by a directed acyclic graph with vertices of the form $(u, t)$ for $1\leq t \leq T$. In addition, this directed graph is a minimal (Definition \ref{def: minimal minus at T-1} below echoes the minimality) graph rooted at $(v, T)$ such that the following holds for each $(u, t)$ in this graph: if $u$ is a real minus at time $t$, then there are at least  two incoming edges from vertices of the form $(w, t-1)$ with $w\in N(v)$ such that
$\sigma_w(t-1) = -1$; otherwise the incoming degree of $(u, t)$ is 0. Provided with this, a natural attempt for our proof is to sum over probabilities for all realizations of the update history, which then entails a rather involved counting problem. To address this, we will prove instead by induction, which is natural since the update history is intrinsically defined via induction. A key challenge here is to formulate an induction hypothesis that is suitably strong, as incorporated in \eqref{eq: minus separate product} below.
Since $\mathbb T_d$ is a bipartite graph, the updates of odd points only depend on the values of the even points right before these updates, and vice versa. Hence, for the purpose of induction, it suffices to consider collections of odd points (and similarly collections of even points).
\begin{thm}\label{thm: minus separate product}
    There exist constants $c_1, c_2, c_3>0$ depending only on $d$ such that the following holds. Let $A$ be an arbitrary collection of odd points, and suppose that $A$ is the disjoint union of non-adjacent odd clusters $A_1, \cdots, A_k$ (for some $k\geq 1$). Then, we have \begin{equation}\label{eq: minus separate product}
        \P(\sigma_{A}(t)=-1)\le \prod_{i=1}^k\Big[(\frac{\eps}{2})^{|A_i|}+c_1\eps^{|A_i|+1} c_2^{\Tri(A_i)}c_3^{\Dtri(A_i)}\Big]
    \end{equation} for some constants $c_1,c_2,c_3>1$ depending only on $d$.
\end{thm}
Clearly, Theorem \ref{thm: minus main higher dimension} is an immediate corollary of Theorem \ref{thm: minus separate product}. As announced earlier, we will prove \eqref{eq: minus separate product} by induction. Since every minus at time 1 is a noise minus (recalling that our initial configuration is all plus), we have $$\P(\sigma_{A}(1)=-1)= (\frac{\eps}{2})^{|A|},$$verifying \eqref{eq: minus separate product} for $t = 1$. From now on, assume \eqref{eq: minus separate product} holds for $t\le T-1$ and we will consider the case when $t=T$. In what follows, we will first consider the case when $k=1$ in Section \ref{sec: single cluster}, which encapsulates most of the conceptual challenge. Then in Section \ref{sec: multiple cluster} we extend it to the case for $k \geq 2$ by addressing the additional complications.
\subsection{The case of a single minus cluster}\label{sec: single cluster}
In this subsection, we assume that $A$ is a fixed odd cluster with $n$ vertices (at times, we will use the notation $\mathsf{A}$ to denote a general set, whereas the font emphasizes the difference from the fixed set $A$).
Our goal is to bound the probability that $\sigma_v(T) = 1$ for all $v\in A$.
\begin{defi}\label{def: pre-real points}
    Let $B$ be a collection of even vertices. For $0 \leq p\leq d$, define $\Xi_p(A, B)$ to be the collection of $v\in A$ with $p$ neighbors in $B$. 
In addition, we define $\Noi(A, B) = \Xi_0(A, B) \cup \Xi_1(A, B)$ and define $\Prr(A, B) = A \setminus \Noi(A, B)$.
\end{defi}
The purpose of Definition \ref{def: pre-real points} is to facilitate the computation for the probability of updating $A$ to all minuses at time $T$ by averaging over possible configurations on $N(A)$ at time $T-1$. To this end, we have
\begin{align}
    &\P(\sigma_A(T)=-1)\nonumber\\=~&\sum_{B\subset N(A)}\P(\sigma_A(T)=\sigma_B(T-1)=-1,\sigma_{N(A)\setminus B}(T-1)=1)\nonumber\\=~&\sum_{C\subset A}\sum_{\substack{B\subset N(A),\\ \Prr(A,B)=C}}\P(\sigma_A(T)=\sigma_B(T-1)=-1,\sigma_{N(A)\setminus B}(T-1)=1)\nonumber\\=~&\sum_{C\subset A}\sum_{\substack{B\subset N(A),\\ \Prr(A,B)=C}}\P(\sigma_B(T-1)=-1,\sigma_{N(A)\setminus B}(T-1)=1)\cdot (\frac{\eps}{2})^{n-|C|}\cdot(1-\frac{\eps}{2})^{|C|},\label{eq: one cluster minus information graph calculation 1}
\end{align}where the last step follows from our updating rules. For the convenience of analysis later, we denote by $\mathbb P_{C \neq \emptyset}(\sigma_A(T) = -1)$ the right-hand side of \eqref{eq: one cluster minus information graph calculation 1} with the modification that the sum is only over $\emptyset \neq C \subset A$. Then clearly, we have that
\begin{equation*}
\mathbb P(\sigma_A(T) = -1) \leq (\epsilon/2)^n +  \mathbb P_{C \neq \emptyset}(\sigma_A(T) = -1)\,.
\end{equation*}
Note that we managed to reduce the probability at time $T$ to a sum over probabilities at time $T-1$ in \eqref{eq: one cluster minus information graph calculation 1}.  In order to efficiently bound the sum above, for each $C \neq \emptyset$ it would be useful to consider some minimal $B$ such that $\Prr(A, B) = C$, as follows.
\begin{defi}\label{def: minimal minus at T-1}
    Recall that $A$ is a collection of odd vertices. Let $\sB(A)$ denote the collection of $B$ such that there does not exist $B'$ satisfying $\Prr(A,B)=\Prr(A,B')$ and $B'\subsetneq B$. Furthermore, for $C\subset A$, let  $\sB(A,C)$ denote the collection of $B\in \sB(A)$ such that $\Prr(A,B)=C$. 
\end{defi}
\begin{lem}\label{lem: minimal minus subset reduction}
    For $C\subset A$, we have \begin{align}
        &\sum_{\substack{B\subset N(A),\\ \Prr(A,B)=C}}\P(\sigma_B(T-1)=-1,\sigma_{N(A)\setminus B}(T-1)=1)\le \sum_{B\in \sB(A,C)}\P(\sigma_B(T-1)=-1).\label{eq: minimal minus subset reduction}
    \end{align}
\end{lem}
Combining Lemma~\ref{lem: minimal minus subset reduction} and \eqref{eq: one cluster minus information graph calculation 1}, we get that \begin{align}
    \mathbb P_{C\neq \emptyset}(\sigma_A(T)=-1)\le\sum_{\emptyset \neq C\subset A}\sum_{B\in \sB(A,C)}\P(\sigma_B(T-1)=-1)\cdot (\frac{\eps}{2})^{n-|C|}\cdot(1-\frac{\eps}{2})^{|C|}.\label{eq: one cluster minus information graph calculation final}
\end{align}
We next analyze the probability on the right-hand side of \eqref{eq: one cluster minus information graph calculation final}. Suppose that $B$ is the union of non-adjacent even clusters $B_1, \cdots, B_r$ (i.e., $\{B_i\}_{1\le i\le k}$ is the collection of the even connected components of $B$). Let 
\begin{equation}\label{eq: def of even cluster parameters}
    b_i=|B_i|,~ s_i=\Tri(B_i) \mbox{ and } t_i=\Dtri(B_i).
\end{equation}
Applying the induction hypothesis \eqref{eq: minus separate product}, we get that \begin{equation}\label{eq: 1cluster prob at T-1 with tri step 1} \begin{aligned}
    \P(\sigma_B(T-1)=-1)\le\prod_{i=1}^r\Big[(\frac{\eps}{2})^{b_i}+c_1\eps^{b_i+1} c_2^{s_i}c_3^{t_i}\Big].
\end{aligned}
\end{equation} In order to analyze the clusters in $B\in \sB(A,C)$, we consider the structure of $\Prr(A,B)$, in the next few lemmas. \begin{defi}\label{def: induced cluster}
    Let $C$ be a collection of odd vertices and let $D=\{x\notin C: |N(x)\cap C|\ge 2\}$ be the collection of even points with at least two neighbors in $C$. We say $D$ is induced by $C$.
    We can similarly define the version with even and odd switched. See Figure~\ref{fig: induced cluster} for an illustration.
\end{defi}
\begin{figure}[ht]
    \centering
    \includegraphics[width=0.5\linewidth]{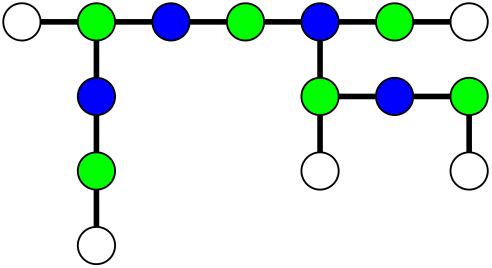}
    \caption{An illustration for induced relation. Green points: the odd cluster $C$. Blue points: the even cluster $D$ induced by $C$.}
    \label{fig: induced cluster}
\end{figure}
\begin{lem}\label{lem: induced cluster}
    Suppose that $C$ is an odd cluster inducing $D$. Then $D$ is connected according to the double-neighboring relation and thus $D$ is an even cluster. By symmetry, the statement holds with odd and even switched.
\end{lem}
\begin{lem}\label{lem: minimality}
    Let $\mathsf{A}$ be an odd cluster. Then for $B\in \mathscr B(\mathsf{A})$ which induces $C$, we have $ C\subset \mathsf{A}$.
\end{lem}
\begin{lem}\label{lem: leaf trifurcation relation}
    Let $\mathsf{A}$ be an odd cluster. Let $B$ be the union of $r$ non-adjacent even clusters. Let $\xi^*(\mathsf{A},B):=\sum_{p=3}^d (p-2) |\Xi_p(\mathsf{A}, B)|$. Then, we have
    \begin{equation}\label{eq: leaf trifurcation relation}
        \xi^*(\mathsf{A},B)+|\Prr(\mathsf{A},B)|\le |B|-r.
    \end{equation}
\end{lem}
Recall that $B_1, \cdots, B_r$ are the even connected components of $B$ and recall \eqref{eq: def of even cluster parameters}.  By Lemma \ref{lem: induced cluster}, we can let $C_i$ be the odd cluster induced by $B_i$. It is clear that $C_i$'s are disjoint. Let $C$ be the set of odd vertices induced by $B$. Since $B_i$'s are non-adjacent, we see that $\cup_{i=1}^rC_i=C$. Thus, by Definition \ref{def: pre-real points} and Lemma~\ref{lem: minimality}, we get that $\cup_{i=1}^r C_i=\Prr(A,B)$. Let $m=|\Prr(A,B)|$, and let $a_i$ be the size of the odd cluster $C_i$. Then $\sum_{i=1}^ra_i=m$. Applying Lemma~\ref{lem: leaf trifurcation relation}, we get that $\sum_{i=1}^r b_i\ge m+\xi^*(A,B)+r$ and $b_i\ge a_i+1$ (the latter can be deduced from Lemma 2.14 with $A = C_i$ and $B = B_i$). Then we get that \begin{equation}
    \prod_{i=1}^r\Big[(\frac{\eps}{2})^{b_i}+c_1\eps^{b_i+1} c_2^{s_i}c_3^{t_i}\Big]\le\eps^{m+r+\xi^*(A,B)}\times\prod_{i=1}^r\Big[(\frac{1}{2})^{a_i+1}+c_1\eps c_2^{s_i}c_3^{t_i}\Big].\label{eq: 1cluster prob at T-1 with tri}
\end{equation}
If $x$ is a type $p$ single trifurcation of $B$ for some $p\geq 3$, we have that $x$ is in the set induced by $B$. Thus, by Lemma~\ref{lem: minimality} we get that $x\in A$ and $x\in \Xi_p(A,B)$, leading to \begin{equation}\label{eq: trifurcation relation 1}
    \begin{aligned}
        \sum_{i=1}^rs_i\le \xi^*(A,B).
    \end{aligned}
\end{equation}
For a type $p~(p\ge 3)$ double trifurcation point $x$ of $B$, there exist $y_1, \cdots, y_p \in N(x)$ such that  there exists $z_i \in B~(1\le i\le p)$ for some $z_i\in N(y_i)\setminus\{x\}$. We claim that $x$ has to be a type $p'$ single trifurcation of $A$ for some $p' \geq p$. (See Figure~\ref{fig: double tri for B} for an illustration.) Otherwise, we have $|N(x) \cap A| \leq p-1$ and without loss of generality we assume $y_1\notin A$. We consider the graph $\hat{G}=G|_{V\setminus\{y_1\}}$. Note that $G$ is a tree and thus $z_1$ is not connected to $z_2$ in $\hat{G}.$
Let $G_1, G_2$ be the two connected components of the graph $\hat{G}$ containing $z_1$ and $z_2$ respectively. Since $y_1$ is an odd vertex and $G$ is a tree, we get that the odd vertices in $G_1$ are not connected to the odd vertices in $G_2$ according to the double-neighboring relation. Since $z_1\in B\cap V(G_1)$ and $z_2\in B\cap V(G_2)$, we get from the minimality of $B$ (recall that $B\in \mathscr B(A, C)$) that $A\cap V(G_1)\neq \emptyset$ and $A\cap V(G_2)\neq \emptyset$. This contradicts the fact that $A$ is an odd cluster (since $G_1$ is not connected to $G_2$ via a double-neighboring relation) and thus completes the proof of our claim. Therefore we get that \begin{equation}\label{eq: trifurcation relation 2}
    \begin{aligned}
        \sum_{i=1}^rt_i\le \Tri(A).
    \end{aligned}
\end{equation}
\begin{figure}[ht]
    \centering
   \subfloat[A double trifurcation of $B$ is a single trifurcation of $A$.]{ \includegraphics[width=0.4\linewidth]{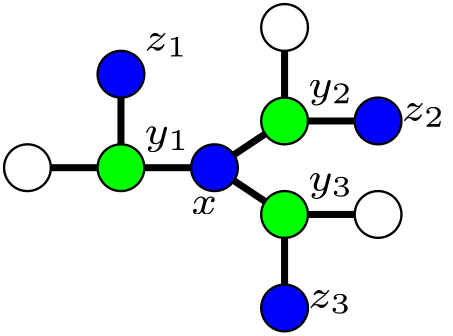}}\hspace{0.15\linewidth}
   \subfloat[There exist two even neighbors of a double trifurcation of $B$ which are disconnected in $A\cup B$.]{ \includegraphics[width=0.4\linewidth]{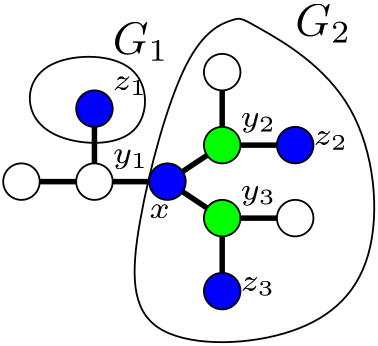}}
    \caption{An illustration of double trifurcations in $B$. Green points: a subset of $A$. Blue points: the even cluster $B$ induced by $A$.}
    \label{fig: double tri for B}
\end{figure}
Note that $c_2,c_3>1$. We get from  \eqref{eq: trifurcation relation 1} and \eqref{eq: trifurcation relation 2} that \begin{align}
    \prod_{i=1}^r\Big[(\frac{1}{2})^{a_i+1}+c_1\eps c_2^{s_i}c_3^{t_i}\Big]&\le c_2^{\sum_{i=1}^rs_i}\cdot c_3^{\sum_{i=1}^rt_i}\cdot\prod_{i=1}^r\Big[(\frac{1}{2})^{a_i+1}+c_1\eps \Big]\nonumber\\&\le c_2^{\xi^*(A,B)}\cdot c_3^{\Tri(A)}\cdot\prod_{i=1}^r\Big[(\frac{1}{2})^{a_i+1}+c_1\eps \Big].\label{eq: 1cluster prob at T-1 with tri step 2}
\end{align}

In order to control the right-hand side of \eqref{eq: 1cluster prob at T-1 with tri step 2} more efficiently, it would be useful to compare $2^{-a_i-1}$ against $c_1 \epsilon$. To this end, for $1\leq i\leq r$, we say a cluster $C_i$ is large if $a_i > -\log_2(c_1\epsilon)-1$, and we say $C_i$ is small otherwise.
Let $r_1, r_2$ denote the number of large and small clusters respectively. Let $a_{1,1},\cdots,a_{1,r_1}$ be the sizes of large clusters and let $m_1=\sum_{i=1}^{r_1}a_{1,i}$; let $a_{2,1},\cdots,a_{2,r_2}$ be the sizes of small clusters and let $m_2=\sum_{i=1}^{r_2}a_{2,i}$. Thus we have the following bound: \begin{align}
   \prod_{i=1}^r\Big[(\frac{1}{2})^{a_i+1}+c_1\eps\Big]&\le \prod_{i=1}^{r_1}2c_1\eps\times\prod_{i=1}^{r_2}(\frac{1}{2})^{a_{2,i}}=2^{r_1-m_2}\times(c_1\eps)^{r_1}.\label{eq: 1cluster prob at T-1 with tri step 3}
\end{align}
Combining \eqref{eq: 1cluster prob at T-1 with tri step 1}, \eqref{eq: 1cluster prob at T-1 with tri}, \eqref{eq: 1cluster prob at T-1 with tri step 2} and \eqref{eq: 1cluster prob at T-1 with tri step 3}  , we derive that\begin{align}
     \P(\sigma_B(T-1)=-1)\le 2^{r_1-m_2}\eps^{m_1+m_2+2r_1+r_2+\xi^*(A,B)}\times c_1^{r_1}\times c_2^{\xi^*(A,B)}\times c_3^{\Tri(A)}.\label{eq: cluster type inequality 1}
\end{align}
Let $\eps>0$ be small enough such that $c_2\eps <1$. Then we get from \eqref{eq: cluster type inequality 1} that \begin{align}
    \P(\sigma_B(T-1)=-1)\le2^{r_1-m_2}\eps^{m_1+m_2+2r_1+r_2}\times c_1^{r_1}\times  c_3^{\Tri(A)}.\label{eq: cluster type inequality}
\end{align}
Since the right-hand side of \eqref{eq: cluster type inequality} depends on $B$ only through $r_1, r_2, m_1$ and $ m_2$, this hugely reduces the complexity of the enumeration problem suggested by \eqref{eq: one cluster minus information graph calculation final}: now we only care about some global features of $\Prr(A, B)$, but do not care about information on the individual cluster $C_i$ in $\Prr(A, B)$.
\subsection{Enumeration of $B$}\label{sec: enumeration overview}
In this subsection, we enumerate the set $B$, for all categories as suggested in \eqref{eq: cluster type inequality}. As in Section~\ref{sec: single cluster}, we fix $A$ to be an odd cluster with $n$ vertices. For each $i$, recall that $C_i$ is the odd cluster induced by the even cluster $B_i$.
\begin{defi}\label{def: G_A}
    For an odd cluster $\mathsf A$, let $B$ be the even cluster induced by $\mathsf A$. Let $G_{\mathsf A}$ be the graph restricted on $\mathsf A\cup  B$, i.e., $G_{\mathsf A}=G|_{\mathsf A\cup B}$. 
\end{defi}
We fix a depth-first search of $G_A$ starting from a leaf vertex of $A$, yielding an order of vertices in $G_A$  according to their first appearance in this search. As a default convention, the order of a set is given by the maximal order over all vertices in this set.
\begin{defi}\label{def: cluster sequence with total sum C}
For $r_1, r_2, m_1, m_2 \geq 0$, we denote by 
$\sC(A,r_1,r_2,m_1,m_2)$ the collection of sequences of disjoint odd clusters $(C_1,\cdots,C_{r_1+r_2})$ satisfying the following properties.\begin{enumerate}
    \item For any $1\le i\le r_1+r_2$, we have $C_i\subset A$ and $C_1, \cdots, C_{r_1+r_2}$ are arranged increasingly according to their order (i.e., the maximal order over all vertices in $C_i$).
    \item There are $r_1$ large clusters with total size  $m_1$ and there are $r_2$ small clusters with total size  $m_2$.
\end{enumerate}
\end{defi}
\begin{defi}\label{def: cluster sequence with total sum B}
     For any sequence of disjoint odd clusters $\mathsf{C}=(C_1,\cdots,C_{r})$ with $\cup_{i=1}^r C_i \subset A$, let $\frB(A,\mathsf{C})$ denote the collection of $B\in\sB(A)$ such that $B$ is a disjoint union of non-adjacent even clusters $B_1, \cdots, B_r$ and that $C_i$ is induced by $B_i$ for all $1\leq i\leq r$.
\end{defi}
\begin{rmk}
    Note that Definitions \ref{def: minimal minus at T-1}, \ref{def: cluster sequence with total sum C} and \ref{def: cluster sequence with total sum B} naturally extend to the case when $A$ is a collection of odd vertices, as we will use in Subsection \ref{sec: multiple cluster}.
\end{rmk}
\begin{rmk}
    For a sequence of disjoint odd clusters $\mathsf{C}=(C_1,\cdots,C_{r})$, we emphasize that $\frB(A,\mathsf{C})$ is not equivalent to $\sB(A,\cup_{i=1}^rC_i)$  in Definition~\ref{def: minimal minus at T-1} because we have posed additional restrictions on the clustering structure of $B$ for $B\in\frB(A,\mathsf{C})$. 
\end{rmk}
Note that $$\bigcup\limits_{\emptyset\neq C\subset A}\sB(A,C)=\bigcup\limits_{\substack{r_1+r_2\ge 1,\\m_1,m_2\ge 0}}~\bigcup\limits_{\mathsf{C}\in \sC(A,r_1,r_2,m_1,m_2)}\frB(A,\mathsf{C}).$$
Combining \eqref{eq: cluster type inequality} with \eqref{eq: one cluster minus information graph calculation final}, we get that $\mathbb P_{C\neq \emptyset}(\sigma_A(T) = -1)$ can be upper-bounded by\begin{align}
    \sum_{\substack{r_1+r_2\ge 1,\\m_1,m_2\ge 0}}\sum_{\mathsf{C}\in \sC(A,r_1,r_2,m_1,m_2)}\sum_{B\in \frB(A,\mathsf{C})}2^{-n+m_1+r_1}\cdot(1-\frac{\eps}{2})^{m_1+m_2}\cdot\eps^{n+2r_1+r_2}\cdot c_1^{r_1}\cdot c_3^{\Tri(A)}.\label{eq: cluster type inequality final}
\end{align} By \eqref{eq: cluster type inequality final}, it suffices to get a uniform upper bound on $|\frB(A,\mathsf{C})|$ for all $\mathsf{C}\in \sC(A,r_1,r_2,m_1,m_2)$ and an upper bound for $|\sC(A,r_1,r_2,m_1,m_2)|$, as incorporated in the next two lemmas. For notational convenience, we write $\chi = \chi(A) = \Tri(A) + \Dtri(A)$ in what follows.

\begin{lem}\label{lem: enumeration induced relation}
For any  sequence $\mathsf{C}=(C_1,\cdots,C_{r_1+r_2})\in \sC(A,r_1,r_2,m_1,m_2)$, we have \begin{equation}\label{eq: enumeration induced relation}
    |\frB(A,\mathsf{C})|\le (d-1)^{2r_1+2r_2+\chi}.
\end{equation}   
\end{lem}
\begin{lem}\label{lem: enumeration lemma final}
Recall that $|A| = n$. We have that $|\sC(A,r_1,r_2,m_1,m_2)|$ is upper-bounded by (below we use the convention that $\binom{a}{b} = 1_{a = b}$ if $b<0$)
     \begin{equation*}\label{eq: enumeration lemma}
     \begin{aligned}
         & 2^{(d+5)\chi}\cdot2^{r_1+r_2+1}\cdot\binom{m_1-1}{r_1-1}\cdot\binom{m_2-1}{r_2-1}\cdot\binom{r_1+r_2}{r_1}\cdot\binom{n-m_1-m_2+r_1+r_2+\chi}{r_1+r_2+\chi} .
     \end{aligned}
\end{equation*}
\end{lem}
\begin{rmk}
    We emphasize that although the definition of $\sC$ depends on the choice of the constant $c_1$ (to separate large and small clusters), the bounds in Lemmas~\ref{lem: enumeration induced relation} and \ref{lem: enumeration lemma final} do not depend on $c_1$.
\end{rmk}
Combining \eqref{eq: one cluster minus information graph calculation final}, \eqref{eq: cluster type inequality final} with Lemmas~\ref{lem: enumeration induced relation} and \ref{lem: enumeration lemma final}, we get that\begin{align}
    &\P_{C\neq \emptyset}(\sigma_A(T)=-1)\nonumber\\
    \le~&\sum_{r_1+r_2\ge 1}\sum_{m_1,m_2\ge 0}\binom{m_1-1}{r_1-1}\cdot\binom{m_2-1}{r_2-1}\cdot\binom{r_1+r_2}{r_1}\cdot2^{-n+m_1+2r_1+r_2+1}\cdot(d-1)^{2r_1+2r_2}\nonumber\\
    &\cdot\binom{n-m_1-m_2+r_1+r_2+\chi}{r_1+r_2+\chi}\cdot \big((d-1)2^{(d+5)}\big)^{\chi}\cdot\eps^{n+2r_1+r_2}\cdot(1-\frac{\eps}{2})^{m_1+m_2}c_1^{r_1}c_3^{\Tri(A)}\nonumber\\
    =~& \sum_{r_1+r_2\ge 1} F(r_1,r_2)\cdot c_1^{r_1}\cdot \binom{r_1+r_2}{r_1}\cdot 2^{2r_1+r_2+1}\cdot(d-1)^{2r_1+2r_2}\cdot \big((d-1)2^{(d+5)}\big)^{\chi}\nonumber\\&\cdot\eps^{n+2r_1+r_2}\cdot c_3^{\Tri(A)}.\label{eq: k=1 with tri final}
\end{align}
Here  $F(r_1, r_2) = F(n, \chi; r_1, r_2)$ denotes the following sum:\begin{align}
    &\sum_{m_1,m_2\ge 0}\binom{n-m_1-m_2+r_1+r_2+\chi}{r_1+r_2+\chi}\cdot\binom{m_1-1}{r_1-1}\cdot\binom{m_2-1}{r_2-1}\cdot 2^{-n+m_1}\cdot(1-\frac{\eps}{2})^{m_1+m_2}.\label{eq: def of F}
\end{align}

\begin{lem}\label{lem: bound for F tri}
    For any $r_1,r_2\ge 0$, we have\begin{equation*}
        F(r_1,r_2)\le 2^{2r_1+2r_2+\chi}\eps^{-r_1}.
    \end{equation*}
    Furthermore, we have the following refined bounds when $r_1 + r_2 = 1$:
    \begin{align*}
        F(1,0)\le 2^{\chi+2},~~\text{and}~~
        F(0,1)\le 2^{\chi+1}.
    \end{align*}
\end{lem}
Now we are ready to prove the case of $k=1$ for \eqref{eq: minus separate product}.
\begin{proof}[Proof of Theorem~\ref{thm: minus separate product}: $k=1$]
    Combining \eqref{eq: k=1 with tri final} and Lemma~\ref{lem: bound for F tri} and the fact that $\binom{r_1+r_2}{r_1}\le 2^{r_1+r_2}$, we get that
\begin{align}
    &\P(\sigma_A(T)=-1)\nonumber\\\le~& (\frac{\eps}{2})^n+ \big((d-1)2^{(d+6)}\big)^{\chi}\cdot c_3^{\Tri(A)}\times\big[32(d-1)^2c_1\eps^{n+2}+8(d-1)^2\eps^{n+1}\nonumber\\& +\sum_{r_1+r_2\ge2}(d-1)^{2r_1+2r_2}\cdot2^{5r_1+4r_2+1}\cdot\eps^{n+r_1+r_2}\big]\nonumber\\=~&(\frac{\eps}{2})^n+ \big((d-1)2^{(d+6)}\big)^{\chi}\cdot c_3^{\Tri(A)}\times\big[32(d-1)^2c_1\eps^{n+2}+8(d-1)^2\eps^{n+1}\nonumber\\&+\frac{7\cdot 2^{9}(d-1)^4\eps^{n+2}-3\cdot 2^{14}(d-1)^6\eps^{n+3}}{(1-32(d-1)^2\eps)(1-16(d-1)^2\eps)}\big].\label{eq: k=1 final calculation with tri}
\end{align} Letting $c_3\ge (d-1)2^{(d+6)}, c_2\ge (d-1)2^{(d+6)}c_3,c_1> 8(d-1)^2$ and $\eps$ small enough (depending only on $c_1, c_2, c_3$ and $d$) completes the proof for the case $k=1$.
\end{proof}

\subsection{The case of multiple minus clusters}\label{sec: multiple cluster}
In this subsection, we consider \eqref{eq: minus separate product} when $k\geq 2$. Suppose that $A$ is the disjoint union of non-adjacent odd clusters $A_1, \cdots, A_k$ (throughout this subsection). Recall the convention on the order for an odd cluster as described preceding to Definition \ref{def: cluster sequence with total sum C}. In this subsection, we fix a depth-first search order which is a concatenation of the order on $G_{A_1}, \cdots, G_{A_k}$. In order to avoid degeneracy, we define $\mathcal J$ to be the collection of $1\leq j\leq k$ such that every $v_j\in A_j$ is a noise minus at time $T$. The key of our proof is to prove the following: \begin{equation}\label{eq: minus separate product for suvivers}
        \mathbb P(\sigma_A(T) = -1; \mathcal J = \emptyset)\le \prod_{i=1}^k\Big[c_1\eps^{|A_i|+1} c_2^{\Tri(A_i)}c_3^{\Dtri(A_i)}\Big].
    \end{equation}
\begin{proof}[Proof of Theorem~\ref{thm: minus separate product} assuming \eqref{eq: minus separate product for suvivers}]
    For any $I\subset [k]$, we define $\cA(I)$ to be the event that $\mathcal J \cap I = \emptyset$ and $\sigma_{A_i}(T) = -1$ for all $i\in I$. By independence of noise updates, we get that \begin{align}
    \P(\sigma_{A}(T)=-1)&=\sum_{I\subset[k]}\mathbb P(\sigma_A(T) = -1; \mathcal J = [k]\setminus I)\nonumber\\&~\le \sum_{I\subset[k]}\P(\cA(I))\cdot (\frac{\eps}{2})^{\sum_{j\notin I}|A_j|}\nonumber\\&\stackrel{\eqref{eq: minus separate product for suvivers}}{\le}\sum_{I\subset[k]}\prod_{i\in I}\Big[c_1\eps^{|A_i|+1} c_2^{\Tri(A_i)}c_3^{\Dtri(A_i)}\Big]\cdot (\frac{\eps}{2})^{\sum_{j\notin I}|A_j|}\nonumber\\&~=\prod_{i=1}^k\Big[(\frac{\eps}{2})^{|A_i|}+c_1\eps^{|A_i|+1} c_2^{\Tri(A_i)}c_3^{\Dtri(A_i)}\Big].
\end{align} Thus we complete the proof of Theorem~\ref{thm: minus separate product}.
\end{proof} 
It remains to prove \eqref{eq: minus separate product for suvivers}, for which a natural attempt is to repeat our proof for the case of $k=1$. However, a more careful analysis is required due to the possibility for two doubly neighboring minus vertices at time $T-1$ to interact with different clusters in $A$ at time $T$.

For $B\subset N(A)$, we continue to let $C$ be the collection of vertices induced by $B$. Instead of having $C \subset A$ when $k=1$ (recall Lemma \ref{lem: minimality}), we need to analyze $C\setminus A$ when $k > 1$, as in the next lemma. 
\begin{lem}\label{lem: extra points}
   Let $B\subset N(A)$ be a collection of even vertices and let $C$ be the collection of vertices induced by $B$. Then for any $x\in C\setminus A$, there exist $1\le i<j\le k$ such that $\dist(A_i,x)=\dist(A_j,x)=2$.
\end{lem}

In light of Lemma \ref{lem: extra points}, we define $W = W(A)$ by \begin{equation}\label{eq: def of extra points}
    W=\{x:\exists 1\le i<j\le k\mbox{ such that } \dist(A_i,x)=\dist(A_j,x)=2\}.
\end{equation}
\begin{lem}\label{lem: number of extra points}
We have $|W|\le k-1.$ 
\end{lem}

For a vertex set $B$, we define $\phi(B)$ to be the collection of vertices $w\in W$ such that $w$ has at least two double neighbors in $B$. Thus, we get that \begin{equation*}
    \begin{aligned}
        &\P(\sigma_{A}(T)=-1\mid \sigma_{B}(T-1)=-1,\sigma_{N(A)\setminus B}(T-1)=1)\\=&\P(\sigma_{A
        \cup \phi(B)}(T)=-1\mid \sigma_{B}(T-1)=-1,\sigma_{N(A)\setminus B}(T-1)=1)\cdot (1-\frac{\eps}{2})^{-|\phi(B)|}.
    \end{aligned}
\end{equation*} Therefore, we obtain that\begin{equation}
    \begin{aligned}\label{eq: prime reduction 1}
        &\P(\sigma_{A}(T)=\sigma_{B}(T-1)=-1,\sigma_{N(A)\setminus B}(T-1)=1)\\=&\P(\sigma_{A\cup\phi(B)}(T)=\sigma_{B}(T-1)=-1,\sigma_{N(A)\setminus B}(T-1)=1)\cdot (1-\frac{\eps}{2})^{-|\phi(B)|}.
    \end{aligned}
\end{equation}
\begin{defi}\label{def: multi cluster connect}
    Let $\sB^*$ denote the collection of sets $B\subset N(A)$ such that $\Prr(A_i,B)\neq\emptyset$ for any $1\le i\le k$. Recall the definition of $\sB(A)$ in Definition~\ref{def: minimal minus at T-1}. For $W^*\subset W$, let $\sB_{W^*}$ denote the collection of sets $B\in \sB(A)\cap\sB^*$ such that $\phi(B)=W^*$.
\end{defi}
Denote by $\mathcal C$ the collection of $C\subset A$ such that $C\cap A_i \neq \emptyset$ for all $1\leq i\leq k$. Recall Definition \ref{def: minimal minus at T-1} (and note that it can naturally be extended to the case when $k > 1$). Combining \eqref{eq: prime reduction 1} with Lemma~\ref{lem: extra points} , we may reproduce our computation as in \eqref{eq: one cluster minus information graph calculation 1}. Indeed, we can get that  \begin{align}
    &\P(\sigma_A(T)=-1;\mathcal J = \emptyset)\nonumber\\\le~&\sum_{W^*\subset W}\sum_{\substack{B\in\sB^*,\\\phi(B)=W^*}}\P(\sigma_A(T)=\sigma_B(T-1)=-1,\sigma_{N(A)\setminus B}(T-1)=1)\nonumber\\\le~&\sum_{W^*\subset W}\sum_{C \in \mathcal C}\sum_{\substack{\Prr(A \cup W,B)=C\cup W^*,\\ B\in\sB^*}}\P(\sigma_{A}(T)=\sigma_B(T-1)=-1,\sigma_{N(A)\setminus B}(T-1)=1) \label{eq: multi cluster calculation 1}
\end{align}
where in the last inequality we applied Lemma \ref{lem: extra points} since it implies that any vertex outside $W$ must be in $A$ if it has at least two neighbors in $B$. Combining \eqref{eq: prime reduction 1} with \eqref{eq: multi cluster calculation 1}, we derive that
\begin{align*}
    &\P(\sigma_A(T)=-1;\mathcal J = \emptyset)\nonumber\\\le~&\sum_{W^*\subset W}(1-\frac{\eps}{2})^{-|W^*|}\cdot\sum_{C \in \mathcal C}\sum_{\substack{\Prr(A \cup W,B)=C\cup W^*,\\ B\in\sB^*}}\nonumber\\&\P(\sigma_{A\cup W^*}(T)=\sigma_B(T-1)=-1,\sigma_{N(A)\setminus B}(T-1)=1)\\=~&\sum_{W^*\subset W}(1-\frac{\eps}{2})^{-|W^*|}\cdot\sum_{C \in \mathcal C}\sum_{\substack{\Prr(A \cup W,B)=C\cup W^*,\\ B\in\sB^*}}\nonumber\\&\P(\sigma_B(T-1)=-1,\sigma_{N(A)\setminus B}(T-1)=1)\cdot (\frac{\eps}{2})^{n-|C|}\cdot(1-\frac{\eps}{2})^{|C|+|W^*|}.\nonumber
\end{align*}

Combined with Lemma~\ref{lem: minimal minus subset reduction}, it yields that \begin{align}
    &\P(\sigma_A(T)=-1;\mathcal J = \emptyset)\le\sum_{W^*\subset W}(1-\frac{\eps}{2})^{-|W^*|}\times\nonumber\\&\sum_{C \in \mathcal C}\sum_{B\in \sB(A \cup W,C\cup W^*)}\P(\sigma_B(T-1)=-1)\cdot (\frac{\eps}{2})^{n-|C|}\cdot(1-\frac{\eps}{2})^{|C|+|W^*|}.\label{eq: prime reduction}
\end{align}
In the right-hand side of \eqref{eq: prime reduction}, we decide not to cancel the term $(1-\frac{\eps}{2})^{-|W^*|}$ for further convenience. Now, in order to bound the sum over $C$ in \eqref{eq: prime reduction} efficiently, we introduce the following definition.
\begin{defi}\label{def: multi cluster connect 1}
    Let $\mathsf{A}$ be an odd cluster and for any $p\ge 3$ and let $S(p)\subset \mathsf{A}$ satisfy that for any $ x\in S(p)$, $x$ is a type $p$ double trifurcation of $\mathsf{A}$.
    In addition, let $\fS=(S(3),\cdots,S(d))$ and $\kappa(\fS)=\sum_{p=3}^d(p-2)|S(p)|$. We use the notation $\sC(\mathsf{A},r_1,r_2,m_1,m_2,\fS)$ to denote the collection of disjoint odd cluster sequences of the form $(C_1,\cdots,C_{r_1+r_2})$ such that the following hold:\begin{enumerate}
        \item $(C_1,\cdots,C_{r_1+r_2})\in \sC(\mathsf{A},r_1,r_2,m_1,m_2)$.
        \item For any $3\le p\le d$, $S(p)\subset\cup_{i=1}^{r_1+r_2}C_i$.
        \item Let $\mathsf{A}_1,\cdots,\mathsf{A}_q$ be the odd connected components in $\mathsf{A}\setminus \cup_{p=3}^dS(p)$. Then we have $\cup_{j=1}^{r_1+r_2}C_j\cap \mathsf{A}_i\neq \emptyset$ for any $1\le i\le q$.
    \end{enumerate}
\end{defi}
The next lemma enables us to use the induction hypothesis for the even clusters in $B$.
\begin{lem}\label{lem: multiple cluster cannot merge}
    For $W^*\subset W$, suppose that $A \cup W^*$ is a disjoint union of non-adjacent odd clusters $A^*_1, \cdots, A^*_{k^*}$.
  For any $B\in \sB_{W^*}$, $B$ can be written as a disjoint union of non-empty even vertex sets $B_1^*, \cdots, B^*_{k^*}$ such that the following hold. \begin{enumerate}
    \item For any $1\le i\le k^*$, we have $B_i^*\in \sB(A_i^*)$.
    \item $B^*_i$ and $B^*_j$ are not adjacent for all $1 \leq i < j \leq k^*$.
    \item For each $1\leq i\leq k^*$, $B^*_i$ induces an odd subset in $A^*_i$.
    \end{enumerate} 
\end{lem} 
As a convention, for $B\in \sB_{W^*}$, we will denote by $B_1^*, \ldots, B_{k^*}^*$ the even vertex sets as in Lemma \ref{lem: multiple cluster cannot merge}. We continue to suppose that $A\cup W^*$ is a disjoint union of non-adjacent odd clusters $A^*_1, \ldots, A^*_{k^*}$.
\begin{lem}\label{lem: C cluster to sequence reduction}
    Recall Definitions~\ref{def: minimal minus at T-1} and \ref{def: cluster sequence with total sum B}. For any sequences $\mathsf{C}_i=(C_{i,1},\cdots,C_{i,l_i})~(1\le i\le k^*)$ such that $\cup_{j=1}^{l_i} C_{i, j} \subset A_i^*$, let $\U(\mathsf{C}_1,\cdots,\mathsf{C}_{k^*})=(C_{i,j})_{1\le i\le k^*,1\le j\le l_i}$. Let $S_i(p)$ be the collection of $ x\in A_i^*\cap W$ (recall \eqref{eq: def of extra points}) such that $x$ is a type $p$ double trifurcation of $A^*_i$ and let $\fS_i=(S_i(3),\cdots,S_i(d))$. Then we have  \begin{equation}\label{eq: C cluster to sequence reduction 1}
        \bigcup\limits_{C \in \mathcal C}\sB(A \cup W,C\cup W^*)\subset\bigcup\limits_{\substack{r_{i,1}+r_{i,2}\ge 1,\\ m_{i,1},m_{i,2}\ge 0,\\ \textsf{ for } 1\leq i\leq k^*}}~\bigcup\limits_{\substack{\mathsf{C}_i\in \sC(A_i^*,r_{i,1},r_{i,2},m_{i,1},m_{i,2},\fS_i),\\\textsf{ for } 1\leq i\leq k^*}}\frB\Big(A\cup W,\U(\mathsf{C}_1,\cdots,\mathsf{C}_{k^*})\Big).
    \end{equation}
    Furthermore, we have \begin{equation}\label{eq: C cluster to sequence reduction 2}
        |\frB(A\cup W,\U(\mathsf{C}_1,\cdots,\mathsf{C}_{k^*}))|\le \prod_{i=1}^{k^*}|\frB(A_i^*,\mathsf{C}_i)|.
    \end{equation}
\end{lem}

By Lemma~\ref{lem: C cluster to sequence reduction}, we get that 
\begin{align}
    &\sum_{C \in \mathcal C}\sum_{B\in \sB(A \cup W,C\cup W^*)}\P(\sigma_B(T-1)=-1)\cdot (\frac{\eps}{2})^{n-|C|}\cdot(1-\frac{\eps}{2})^{|C|+|W^*|}\nonumber
    \\\le~&\sum_{\substack{r_{i,1}+r_{i,2}\ge 1,\\ m_{i,1},m_{i,2}\ge 0,\\ \textsf{ for } 1\leq i\leq k^*}}\sum_{\substack{\mathsf{C}_i\in \sC(A_i^*,r_{i,1},r_{i,2},m_{i,1},m_{i,2},\fS_i),\\\textsf{ for } 1\leq i\leq k^*}}\sum_{B\in\frB(A\cup W,\U(\mathsf{C}_1,\cdots,\mathsf{C}_{k^*}))}\P(\sigma_{B}(T-1)=-1)\nonumber\\&\cdot (\frac{\eps}{2})^{n-|C|}\cdot(1-\frac{\eps}{2})^{|C|+|W^*|}.\label{eq: merge clusters enumeration 0}
\end{align}
We wish to use our induction hypothesis together with Lemma \ref{lem: multiple cluster cannot merge} to upper-bound $\P(\sigma_B(T-1)=-1)$. Let $b_{i,j}=|B_{i,j}|,~s_{i,j}=\Tri(B_{i,j})$ and $t_{i,j}=\Dtri(B_{i,j})$. 
Then, applying our induction hypothesis \eqref{eq: minus separate product} we get that \begin{equation}\label{eq: multi cluster expansion}
    \P(\sigma_B(T-1)=-1)\le \prod_{i=1}^{k^*}\prod_{j=1}^{l_i}\Big[(\frac{\eps}{2})^{b_{i,j}}+c_1\eps^{b_{i,j}+1}c_2^{s_{i,j}}c_3^{t_{i,j}}\Big].
\end{equation}
Next, we want to bound the right-hand side of \eqref{eq: multi cluster expansion} using the same strategy as we did for \eqref{eq: 1cluster prob at T-1 with tri step 1}. Let $c_{i,j}=|C_{i,j}|$. Then by \eqref{eq: 1cluster prob at T-1 with tri}, we get that \begin{align}
    &\prod_{j=1}^{l_i}\Big[(\frac{\eps}{2})^{b_{i,j}}+c_1\eps^{b_{i,j}+1}c_2^{s_{i,j}}c_3^{t_{i,j}}\Big]\nonumber\\\le~&\eps^{m_{i,1}+m_{i,2}+r_{i,1}+r_{i,2}+\xi^*(A_i^*,B_i^*)}\times\prod_{j=1}^{r_{i,1}+r_{i,2}}\Big[(\frac{1}{2})^{c_{i,j}+1}+c_1\eps c_2^{s_{i,j}}c_3^{t_{i,j}}\Big].\label{eq: multi cluster expansion step 2}
\end{align}
Applying \eqref{eq: trifurcation relation 1}, \eqref{eq: trifurcation relation 2} and \eqref{eq: 1cluster prob at T-1 with tri step 2} to the last term in \eqref{eq: multi cluster expansion step 2}, we get that\begin{align}
    \prod_{j=1}^{r_{i,1}+r_{i,2}}\Big[(\frac{1}{2})^{c_{i,j}+1}+c_1\eps c_2^{s_{i,j}}c_3^{t_{i,j}}\Big]&\le c_2^{\xi^*(A_i^*,B_i^*)}\cdot c_3^{\Tri(A_i^*)}\cdot\prod_{i=1}^{r_{i,1}+r_{i,2}}\Big[(\frac{1}{2})^{c_{i,j}+1}+c_1\eps \Big].\label{eq: multi cluster expansion step 3}
\end{align}
Next, we split the clusters $C_{i,j}$ into large and small clusters according to the comparison between $c_{i,j}$ and $-\log_2(c_1\eps)-1$ (in the sense that $C_{i,j}$ is a large cluster when $c_{i, j}$ is the larger one). Then we get from \eqref{eq: 1cluster prob at T-1 with tri step 3}
\begin{align}
   \prod_{i=1}^{r_{i,1}+r_{i,2}}\Big[(\frac{1}{2})^{c_{i,j}+1}+c_1\eps\Big]&\le 2^{r_{i,1}-m_{i,2}}\times(c_1\eps)^{r_{i,1}}.\label{eq: multi cluster expansion step 4}
\end{align}
Combining \eqref{eq: multi cluster expansion}, \eqref{eq: multi cluster expansion step 2}, \eqref{eq: multi cluster expansion step 3} and \eqref{eq: multi cluster expansion step 4}, we get that
\begin{align}
    &\P(\sigma_B(T-1)=-1)\nonumber\\\le~& \prod_{i=1}^{k^*}2^{r_{i,1}-m_{i,2}}\times\eps^{m_{i,1}+m_{i,2}+2r_{i,1}+r_{i,2}+\xi^*(A_i^*,B_i^*)}\cdot c_1^{r_{i,1}}\cdot c_2^{\xi^*(A_i^*,B_i^*)}\cdot c_3^{\Tri(A_i^*)}.\label{eq: multi cluster induction hypothesis}
\end{align}
\begin{lem}\label{lem: dtri and cluster number relation}
Let $\mathsf{A}$ be an odd cluster and let $S(p)\subset \mathsf{A}$ satisfy that for any $ x\in S(p)$, $x$ is a type $p$ double trifurcation of $\mathsf{A}$. Recall $\fS=(S(3),\cdots,S(d))$ and recall that $\kappa(\fS)=\sum_{p=3}^d (p-2)|S(p)|$.
    Let $\mathsf{C}\in \sC(\mathsf{A},r_1,r_2,m_{1},m_{2},
    \fS)$ and $B\subset \mathfrak{B}(\mathsf{A},\mathsf{C})$. Then we have $r_1+r_2+\xi^*(\mathsf{A},B)\ge \kappa(\fS)+1$.
\end{lem}
Combining \eqref{eq: multi cluster induction hypothesis}, Lemma~\ref{lem: dtri and cluster number relation} and the fact that $c_2\eps\le 1$, we get that
\begin{equation*}\label{eq: multi cluster induction hypothesis new}
    \P(\sigma_B(T-1)=-1)\le \prod_{i=1}^{k^*}2^{r_{i,1}-m_{i,2}}\cdot\eps^{m_{i,1}+m_{i,2}+2r_{i,1}+r_{i,2}}\cdot c_1^{r_{i,1}}\cdot (c_2\eps)^{\max\{\kappa(\fS_i)+1-r_{i,1}-r_{i,2},0\}}\cdot c_3^{\Tri(A_i^*)}.
\end{equation*}
Combined with \eqref{eq: merge clusters enumeration 0}, it yields that
\begin{align}
    &\sum_{C \in \mathcal C}\sum_{B\in \sB(A \cup W,C\cup W^*)}\P(\sigma_B(T-1)=-1)\cdot (\frac{\eps}{2})^{n-|C|}\cdot(1-\frac{\eps}{2})^{|C|+|W^*|}\nonumber
    \\
    \le~&\sum_{\substack{r_{i,1}+r_{i,2}\ge 1,\\ m_{i,1},m_{i,2}\ge 0,\\ \textsf{ for } 1\leq i\leq k^*}}\sum_{\substack{\mathsf{C}_i\in \sC(A_i^*,r_{i,1},r_{i,2},m_{i,1},m_{i,2},\fS_i),\\ \textsf{ for } 1\leq i\leq k^*}}\sum_{B\in\sB(A\cup W,\U(\mathsf{C}_1,\cdots,\mathsf{C}_{k^*}))}\prod_{i=1}^{k^*}2^{-|A_i^*|+m_{i,1}+r_{i,1}}\nonumber\\&\times\eps^{|A_i^*|+2r_{i,1}+r_{i,2}}\times(1-\frac{\eps}{2})^{m_{i,1}+m_{i,2}}\times c_1^{r_{i,1}}\times (c_2\eps)^{\max\{\kappa(\fS_i)+1-r_{i,1}-r_{i,2},0\}}\times c_3^{\Tri(A_i^*)}.\label{eq: merge clusters enumeration 1}
\end{align} Note that the summand on the right-hand side of \eqref{eq: merge clusters enumeration 1} depends on $B$ only through the set $\{r_{i, j}, m_{i, j}\}_{1\le i\le k^*,1\le j\le l_i}$. Furthermore, by Lemma~\ref{lem: C cluster to sequence reduction} (in particular, \eqref{eq: C cluster to sequence reduction 2}), we get from \eqref{eq: merge clusters enumeration 1} that
\begin{align}
    &\sum_{C \in \mathcal C}\sum_{B\in \sB(A \cup W,C\cup W^*)}\P(\sigma_B(T-1)=-1)\cdot (\frac{\eps}{2})^{n-|C|}\cdot(1-\frac{\eps}{2})^{|C|+|W^*|}\nonumber\\ \le~&\sum_{\substack{r_{i,1}+r_{i,2}\ge1,\\ m_{i,1},m_{i,2}\ge 0,\\\textsf{ for } 1\leq i\leq k^*}}\sum_{\substack{\mathsf{C}_i\in \sC(A_i^*,r_{i,1},r_{i,2},m_{i,1},m_{i,2},\fS_i),\\ \textsf{ for } 1\leq i\leq k^*}}\prod_{i=1}^{k^*}|\frB(A_i^*,\mathsf{C}_i)|\cdot\prod_{i=1}^{k^*}2^{-|A_i^*|+m_{i,1}+r_{i,1}}\nonumber\\&\times\prod_{i=1}^{k^*}\eps^{|A_i^*|+2r_{i,1}+r_{i,2}}\times(1-\frac{\eps}{2})^{m_{i,1}+m_{i,2}}\times c_1^{r_{i,1}}\times (c_2\eps)^{\max\{\kappa(\fS_i)+1-r_{i,1}-r_{i,2},0\}}\times c_3^{\Tri(A_i^*)}\nonumber\\
    =~&\prod_{i=1}^{k^*}M_i.\label{eq: merge clusters enumeration}
\end{align}
where $M_i$ is defined by \begin{align}
    M_i=&\sum_{\substack{r_1+r_2\ge 1,\\ m_{1},m_{2}\ge 0}}\sum_{\substack{\mathsf{C}_i\in\sC(A_i^*,r_1,r_2,m_1,m_2,\fS_i) }}|\frB(A_i^*,\mathsf{C}_i)|\cdot2^{-|A_i^*|+m_{1}+r_1}\nonumber\\&\times\eps^{|A_i^*|+2r_1+r_2}\times(1-\frac{\eps}{2})^{m_{1}+m_{2}}\times c_1^{r_1}\times (c_2\eps)^{\max\{\kappa(\fS_i)+1-r_1-r_2,0\}}\times c_3^{\Tri(A_i^*)}\label{eq: specific cluster in merging A}
\end{align}for each $1\le i\le k^*$. We want to repeat our analysis in Section~\ref{sec: enumeration overview} and get a conclusion similar to  \eqref{eq: k=1 final calculation with tri}.

For notation clarity, let $\chi_i=\chi(A_i^*)$ and $n_i = |A^*_i|$. Define $F_i(r_1, r_2) = F(n_i, \chi_i; r_1, r_2)$ as in \eqref{eq: def of F}. Combining \eqref{eq: specific cluster in merging A} with Lemmas~\ref{lem: enumeration induced relation} and \ref{lem: enumeration lemma final}, we get that\begin{align}
    M_i\le~&\sum_{\substack{r_1+r_2\ge 1}}F_i(r_1,r_2)\big((d-1)2^{d+5}\big)^{\chi_i}\cdot \binom{r_1+r_2}{r_1}\cdot (d-1)^{2r_1+2r_2}\cdot 2^{2r_1+r_2+1}\cdot\eps^{|A_i^*|+2r_1+r_2}\nonumber\\&\times  c_1^{r_1}\times(c_2\eps)^{\max\{\kappa(\fS_i)+1-r_1-r_2,0\}}\times c_3^{\Tri(A_i^*)}.\label{eq: first bound for Mi}
\end{align} Combined with Lemma~\ref{lem: bound for F tri}, it yields that \begin{align}
M_i\le~&\sum_{\substack{r_1+r_2\ge 1}}2^{2r_1+2r_2+\chi_i}\eps^{-r_1}\big((d-1)2^{d+5}\big)^{\chi_i}\cdot \binom{r_1+r_2}{r_1}\cdot (d-1)^{2r_1+2r_2}\cdot 2^{2r_1+r_2+1}\nonumber\\
&\times\eps^{|A_i^*|+2r_1+r_2}\times c_1^{r_1}\times(c_2\eps)^{\max\{\kappa(\fS_i)+1-r_1-r_2,0\}}\times c_3^{\Tri(A_i^*)}\nonumber\\
\le~&\sum_{1\le r_1+r_2\le \kappa(\fS_i)}\eps^{|A_i^*|+\kappa(\fS_i)+1}\cdot 2^{5r_1+4r_2+1}\cdot c_1^{r_1}\cdot c_2^{\kappa(\fS_i)+1-r_1-r_2}\cdot \big((d-1)2^{d+6}c_3\big)^{ \Tri(A_i^*)}\nonumber\\
&\times \big((d-1)2^{d+6}\big)^{\Dtri(A_i^*)}\nonumber\\
&+\sum_{r_1+r_2\ge \kappa(\fS_i)+1}\eps^{|A_i^*|+r_1+r_2}\cdot 2^{5r_1+4r_2+1} \cdot \big((d-1)2^{d+6}c_3\big)^{\Tri(A_i^*)}\cdot \big((d-1)2^{d+6}\big)^{\Dtri(A_i^*)}.\nonumber
\end{align} 
Thus in conclusion, there exists an absolute constant $C_1>1$ such that \begin{align}
    M_i&\le C_1^{\kappa(\fS_i)+1} (c_1 c_2)^{\kappa(\fS_i)+1}\cdot \big((d-1)2^{d+6}c_3\big)^{\Tri(A_i^*)}\cdot \big((d-1)2^{d+6}\big)^{\Dtri(A_i^*)}\cdot \eps^{|A_i^*|+\kappa(\fS_i)+1}.\label{eq: multi cluster induction calculation}
\end{align}
In order to give a better control on the special case $\cup_{p=3}^dS_i(p)=\emptyset$, note that the right-hand side of \eqref{eq: first bound for Mi} is the same as the right-hand side of \eqref{eq: k=1 with tri final}. Thus we get from the calculation in \eqref{eq: k=1 final calculation with tri} that there exists a constant $C_2>1$ depending only on $d$ such that (we also assume $c_1\eps<1$)
\begin{align}
    M_i&\le \big((d-1)2^{(d+6)}\big)^{\chi_i}\cdot c_3^{\Tri(A_i^*)}\times\big[32(d-1)^2c_1\eps^{|A_i^*|+2}+8(d-1)^2\eps^{|A_i^*|+1}\nonumber\\&+\frac{7\cdot 2^{9}(d-1)^4\eps^{|A_i^*|+2}-3\cdot 2^{14}(d-1)^6\eps^{|A_i^*|+3}}{(1-32(d-1)^2\eps)(1-16(d-1)^2\eps)}\big]\nonumber\\&\le \big((d-1)2^{(d+6)}c_3\big)^{\Tri(A_i^*)}\cdot \big((d-1)2^{(d+6)}\big)^{\Dtri(A_i^*)}\cdot C_2\eps^{|A_i^*|+1}.\label{eq: multi cluster induction calculation 2}
\end{align}
Combining \eqref{eq: merge clusters enumeration} with \eqref{eq: multi cluster induction calculation} and \eqref{eq: multi cluster induction calculation 2}, we get that \begin{align}
    &\sum_{C \in \mathcal C}\sum_{B\in \sB(A\cup W,C\cup W^*)}\P(\sigma_B(T-1)=-1)\cdot (\frac{\eps}{2})^{n-|C|}\cdot(1-\frac{\eps}{2})^{|C|+|W^*|}\nonumber\\
    \le~&\prod_{i=1}^{k^*} C_2\cdot\big((d-1)2^{(d+6)}c_3\big)^{\Tri(A_i^*)}\cdot \big((d-1)2^{(d+6)}\big)^{\Dtri(A_i^*)}\cdot \eps^{|A_i^*|+\kappa(\fS_i)+1}\times \prod_{\kappa(\fS_i)\ge 1}(c_1c_2C_1)^{\kappa(\fS_i)+1}\nonumber\\
=~&C_2^{k^*}\eps^{k^*+\sum_{i=1}^{k^*}(|A_i^*|+\kappa(\fS_i))}\cdot (c_1c_2C_1)^{\sum_{\kappa(\fS_i)\ge 1}(\kappa(\fS_i)+1)}\cdot \big((d-1)2^{(d+6)}c_3\big)^{ \sum_{i=1}^{k^*}\Tri(A_i^*)}\nonumber\\&\times\big((d-1)2^{(d+6)}\big)^{\sum_{i=1}^{k^*}\Dtri(A_i^*)}.\label{eq: merged cluster bound collect}
\end{align}
Next, we bound the exponents in \eqref{eq: merged cluster bound collect}. 
\begin{lem}\label{lem: bound for the exponents in the multi cluster case}
    Let $A$ be a collection of odd vertices and let $W$ be defined as in \eqref{eq: def of extra points}. Let $W^*\subset W$ be a subset of $W$. Recall that $A_1^*,\cdots,A_{k^*}^*$ are the odd connected components of $A\cup W^*$. Let $S_i(p)$ be the collection of $ x\in A_i^*\cap W$ such that $x$ is a type $p$ double trifurcation of $A^*_i$ and let $\mathsf{S}_i=(S_i(3),\cdots,S_i(d))$. Recall that $\kappa(\mathsf{S}_i) = \sum_{p=3}^d (p-2)|S_i(p)|$. Then we have the following inequalities:
    \begin{equation}\label{eq: cluster number edge relation}
   k^*+\sum_{i=1}^{k^*}(|A_i^*|+\kappa(\fS_i))=|A|+ k^*+|W^*|+\sum_{i=1}^{k^*}\kappa(\fS_i)=|A|+k,
\end{equation}
 \begin{equation}\label{eq: cluster merge trifurcation relation}
    \sum_{i=1}^{k^*}\Tri(A_i^*)\le dk+\sum_{i=1}^{k}\Tri(A_i)~~\text{and}~~\sum_{i=1}^{k^*}\Dtri(A_i^*)\le (d^2-d+1)k+\sum_{i=1}^{k}\Dtri(A_i),
\end{equation} 
 \begin{equation}\label{eq: dtri in extra point number bound}
    \sum_{i=1}^{k^*}(\kappa(\fS_i)+1)\le \frac{(d-1)k}{d}.
\end{equation}
\end{lem}

Combining \eqref{eq: merged cluster bound collect} with \eqref{eq: cluster number edge relation}, \eqref{eq: cluster merge trifurcation relation} and \eqref{eq: dtri in extra point number bound}, we get that \begin{align}
    &\sum_{C \in \mathcal C}\sum_{B\in \sB(A\cup W,C\cup W^*)}\P(\sigma_B(T-1)=-1)\cdot (\frac{\eps}{2})^{n-|C|}\cdot(1-\frac{\eps}{2})^{|C|+|W^*|}\nonumber\\\le~& (C_2\cdot (d-1)^{d^2+1}2^{(d^2+1)(d+6)}c_3^d)^{k}\eps^{|A|+k}\big((d-1)2^{d+6}c_3\big)^{\sum_{i=1}^{k}\Tri(A_i)}\nonumber\\
    &\times\big((d-1)2^{d+6}\big)^{\sum_{i=1}^{k}\Dtri(A_i)}\cdot (c_1c_2C_1)^{\frac{(d-1)k}{d}}.
    \label{eq: merge clusters enumeration final}
\end{align}

Recall Lemma \ref{lem: number of extra points}. Combining \eqref{eq: prime reduction} and \eqref{eq: merge clusters enumeration final}, we derive that
\begin{align}
&\P(\sigma_A(T)=-1)\nonumber\\\le~&\sum_{W^*\subset W}(1-\frac{\eps}{2})^{-|W^*|}\cdot(C_2\cdot (d-1)^{d^2+1}2^{(d^2+1)(d+6)}c_3^d)^{k}\eps^{|A|+k}\big((d-1)2^{d+6}c_3\big)^{\sum_{i=1}^{k}\Tri(A_i)}\nonumber\\&\cdot \big((d-1)2^{d+6}c_3\big)^{\sum_{i=1}^{k}\Dtri(A_i)}\cdot (c_1c_2)^{\frac{(d-1)k}{d}}\nonumber\\\le~& 4^k\cdot(C_3\cdot (d-1)^{d^2+1}2^{(d^2+1)(d+6)}c_3^d)^{k}\eps^{|A|+k}\big((d-1)2^{d+6}c_3\big)^{\sum_{i=1}^{k}\Tri(A_i)}\nonumber\\
&\times\big((d-1)2^{d+6}\big)^{\sum_{i=1}^{k}\Dtri(A_i)}\cdot (c_1c_2C_1)^{\frac{(d-1)k}{d}}.
\end{align}
Letting $c_2\ge (d-1)2^{d+6}c_3$, $c_3\ge (d-1)2^{d+6}$ and $c_1\ge \big(4C_3\cdot (d-1)^{d^2+1}2^{(d^2+1)(d+6)}c_3^d\big)^dc_2^{d-1}C_1^{d-1}$ and choose $\eps>0$ small enough (depending only on $d,c_1,c_2$ and $c_3$) completes the proof of \eqref{eq: minus separate product for suvivers}.

\section{Proof of Lemmas \ref{lem: enumeration induced relation} and \ref{lem: enumeration lemma final}}\label{sec: enumeration lemma proof}
\subsection{Proof of Lemma \ref{lem: enumeration induced relation}}
In this section, we prove Lemma \ref{lem: enumeration induced relation}.
Recall the definition of $G_A$ in Definition~\ref{def: G_A}. We start with a lemma that upper-bounds the number of branching vertices in $G_A$. Let $\Xi_p(G_A)$ denote the collection of degree $p$ vertices in $G_A$ and let $\xi^*(G_A)=\sum_{p=3}^d(p-2)|\Xi_p(G_A)|$. Recall that $\chi(A) = \Tri(A) + \Dtri(A)$
\begin{lem}\label{lem: d=3 trifurcation relation}
    Let $A$ be an odd cluster. Then we have $\xi^*(G_A)\le \chi(A)$. 
\end{lem}
\begin{proof}
Recalling Definition~\ref{def: G_A}, we get that the collection of odd vertices in $V(G_A)$ is $A$ and the collection of even vertices in $V(G_A)$ is $B$ where $B$ is the collection of even vertices induced by $A$.
    Let $x$ be a degree $p~(p\ge 3)$ vertex in $G_A$ and let $y_1,\cdots,y_p$ be the neighbors of $x$ in $G_A$. If $x$ is an odd vertex, then we get that $y_i\in B$ and thus $N(y_i)\cap A\setminus\{x\}\neq \emptyset$. Therefore, $x$ is a type $p$ double trifurcation of $A$. 
    
    If $x$ is an even vertex, then $y_1,\cdots,y_p$ are odd vertices in $G_A$ and thus $y_1,\cdots,y_p\in A$. Therefore, $N(x)\subset A$ and we get that $x$ is a type $p$ single trifurcation of $A$. As a result, for any degree $p$ vertex $x$ in $G_A$, it is either a type $p$ single trifurcation or a type $p$ double trifurcation of $A$. Therefore we get that $\xi^*(G_A)\le \chi(A)$.
\end{proof}
\begin{lem}\label{lem: induced relation}
    Let $B\subset V(G)$ be an even cluster and let $C$ be the odd cluster induced by $B$. In addition, let $D$ be the even cluster induced by $C$.
    Then we have $D\subset B$. 
\end{lem}
\begin{proof}
    If there exists a vertex $x\in D\setminus B$, let $y,z$ be two neighbors of $x$ in $C$. Since $x\notin B$ and $G$ is a tree, the vertices in $N(y)\cap B$ and the vertices in $N(z)\cap B$ are disconnected through the double-neighboring relation in $B$. Thus we get a contradiction to the fact that $B$ is an even cluster. 
\end{proof}
\begin{lem}\label{lem: induced relation enumeration bound}
       For an odd cluster $C$, let $D$ be the even cluster induced by $C$ and let $H$ be the graph restricted on $D\cup C$, i.e., $H=G|_{D\cup C}$. Define $\tilde{\sB}(C)$ to be the collection of even clusters $B$ such that the following hold.\begin{enumerate}
        \item $C$ is induced by $B$.
        \item There does not exist an even cluster $B'$ such that $C$ is induced by $B'$ and $B'\subsetneq B$.
    \end{enumerate} Then we have that $|\tilde{\sB}(C)|\le (d-1)^{2+\xi^*(H)}$. 
\end{lem}
\begin{proof}
By Lemma~\ref{lem: induced relation}, we get that $D\subset B$. It suffices to consider $B\setminus D$.
Let $L$ denote the collection of leaves in $H$, and for any $x\in L$ let $M_x=N(x)\setminus D$. Note that $L\subset C$ since any vertex in $D$ has at least two neighbors in $C$ by definition. We define a map $\phi$ from $\tilde{\sB}(C)$ to $\bigotimes_{x\in L}M_x$, as follows (so each $\phi(B)$ is a vector $(\phi(B, x))_{x\in L}$). For any $B\in \tilde{\sB}(C)$ and $x\in L$, let $\phi(B,x)$  be the vertex with the smallest order in $M_x\cap B$; this is well-defined since from $x\in C$ we have  $|N(x)\cap B|\ge 2$, $|N(x)\cap D|=1$ and thus $M_x\cap B\neq \emptyset$. We now prove that $\phi$ is injective. To this end, let $M_\phi = \cup_{x\in L} \phi(B, x)$ and let $B'= M_\phi \cup D$. By the definition of $\phi$, we have that $B'\subset B$. Furthermore, for any $x\in L$, it has one neighbor in $D$ since otherwise $C$ cannot be an odd cluster. In addition, we have added at least another neighbor of $x$ to $B'$. Thus, $x$ has at least two neighbors in $B'$  (this obviously holds for $x\in C\setminus L$). This implies that $B'$ induces a set that contains $C$. Hence by the minimality of $B$, we get that $B=B'$, and therefore we get that $\phi$ is injective. Thus, $|\tilde{\sB}(C)|$ is upper-bounded by the number of images of $\phi$. Since $|M_x|\leq d-1$, we then see that $|\tilde{\sB}(C)| \leq (d-1)^{|L|}$. Note that $H$ is a tree with maximal degree $d$ and thus $|L| \leq 2 +\sum_{p=3}^d(p-2)|\Xi_p(H)|=2+\xi^*(H)$. Hence we complete the proof of Lemma~\ref{lem: induced relation enumeration bound}.
\end{proof}
\begin{proof}[Proof of Lemma \ref{lem: enumeration induced relation}]
    Fix a sequence of disjoint odd clusters $\mathsf{C}=(C_1,\cdots,C_{r_1+r_2})$ and let $B\in \frB(A,\mathsf{C})$. Suppose that $B$ is a disjoint union of non-adjacent even clusters $B_1, \ldots, B_{r_1 +r_2}$ where $B_i$ induces $C_i$. 
    Let $D_i$ be the even cluster induced by $C_i$ and let $H_i$ denote the graph restricted on $C_i\cup D_i$.
    By the minimality of $B$, we get that $B_i\in \tilde{\sB}(C_i)$.  Then by Lemmas~\ref{lem: d=3 trifurcation relation} and \ref{lem: induced relation enumeration bound}, $$|\frB(A,\mathsf{C})| \leq \prod_{i=1}^{r_1+r_2}(d-1)^{2+|\xi^*(H_i)|}\le (d-1)^{2r_1+2r_2+\xi^*(G_A)}\le (d-1)^{2r_1+2r_2+\chi(A)},$$ completing the proof of the lemma by noting that $\sum_{i=1}^{r_1+r_2}\xi^*(H_i) \leq \xi^*(G_A) $ since $\cup V(H_i) \subset V(G_A)$.
\end{proof}
\subsection{Proof of Lemma~\ref{lem: enumeration lemma final}: A simple case}
In order to prove Lemma \ref{lem: enumeration lemma final}, we first consider a special case: $A$ is an odd line and we do not distinguish the large and small odd clusters. We start with some definitions.
\begin{defi}
    Let $A$ be an odd cluster and fix a depth-first search of $A$ starting from a leaf vertex. As before we give the vertices in $A$ an order according to their first appearance in the depth-first search.
Let $\tilde\sC(A,r,m)$ denote the collection of  sequences of disjoint odd clusters $(C_1,\cdots,C_{r})$ satisfying the following properties.\begin{enumerate}
\item $\sum_{i=1}^r|C_i|=m$.
    \item For any $1\le i\le r$ we have $C_i\subset A$, and for any $i< j$ the order of $C_i$ is smaller than the order of $C_j$.
\end{enumerate}
\end{defi}

\begin{defi}\label{def: white black cluster}
    For $(C_1, \ldots, C_r) \in \tilde\sC(A, r, m)$, we define $D_i$ to be the even cluster induced by $C_i$. We call each $C_i\cup D_i$ a white component and we call the connected components in $V(G_A)\setminus\cup_{i=1}^r(C_i\cup D_i)$ black components. For either a white or black component, we define its \textbf{capacity} to be the number of odd points in the component.
\end{defi}
\begin{rmk}
    In Definition~\ref{def: white black cluster}, it is possible that a black component has capacity $0$ because possibly it only contains even vertices.
\end{rmk}

\begin{lem}\label{lem: enumeration lemma notri}
    Suppose that $A$ is an odd line with $n$ vertices. 
Then we have \begin{equation}\label{eq: enumeration lemma no tri}
    |\tilde\sC(A,r,m)|\le 2\cdot\binom{m-1}{r-1}\binom{n-m+r}{r}.
\end{equation}
\end{lem}
\begin{proof}
Since black and white components appear alternatively and since all white components start and end at odd vertices, if we know the capacities of white and black components along the line (i.e., in the depth-first search order) and the color of the first component, then we can recover all white components and as a result recover $(C_1, \ldots, C_r)$.  Now, note that the number of ways to choose $|C_1|, \ldots, |C_r|$ is the same as the number of ways to split $m$ into $r$ positive integers and thus is $\binom{m-1}{r-1}$.
In addition, there can be at most $r+1$ black components, and the number of ways to choose their capacities is upper-bounded by the number of ways to split $n-m$ into $r+1$ non-negative integers, which is $\binom{n-m+r}{r}$. Altogether, this completes the proof.
\end{proof}
We next generalize Lemma \ref{lem: enumeration lemma notri} to the case when $A$ is an odd cluster.
\begin{prop}\label{prop: enumeration lemma step 1}
Suppose that $A$ is an odd cluster with $n$ vertices.
 Then we have \begin{equation}\label{eq: enumeration lemma step 1}
    |\tilde\sC(A,r,m)|\le \binom{m+r-1}{r-1}\cdot\binom{n-m+r+\xi^*(G_A)}{r+\xi^*(G_A)}\cdot 2^{\xi^*(G_A)}.
\end{equation}
\end{prop}
The proof of Proposition \ref{prop: enumeration lemma step 1} is inspired by that of Lemma \ref{lem: enumeration lemma notri}, for which we wish to construct a line graph out of $G_A$ as follows.

\begin{defi}\label{def: psi}
Let $\mathsf{C}=(C_1,\cdots,C_r)$ be a given sequence of disjoint odd clusters. 

First, we employ the following \textbf{leaf-cutting procedure}. Let $G_A^0=G_A$. Inductively for each $t\geq 0$, we let $x$ be the first black leaf of $G_A^t$ and let $B_x$ be the maximal black cluster containing $x$ but not containing a branching vertex in $G_A^t$; then we set  $G_A^{t+1}=G_A^t\setminus B_x$. We stop this inductive procedure until all leaves in $G^s_A$ are white (say at some time $s$), and we denote $G^*_A = G^s_A$. In addition, let $\psi_1(\mathsf{C})$ be the sequence of capacities for the clusters we removed for the leaves in $G_A$. Here we use the convention that if $v$ is a white leaf in $G_A$, then the value for $v$ in $\psi_1(\mathsf{C})$ is 0.

Next, we wish to construct a map $\Psi$ from $G_A^*$ to a line $\tilde G_A$, which preserves the color and the parity of a vertex. 
To this end, we let $\tilde G_A$ be a line graph with the same number of vertices in $G^*_A$, and order the vertices in $\tilde G_A$ by the depth-first search thereon. Then naturally we can let $\Psi$  map each vertex $u\in V(G^*_A)$ to a vertex $v\in V(\tilde G_A)$ such that their orders in respective graphs are identical.  For $v\in V(\tilde G_A)$, we say it is black (white) in $\Tilde{G}_A$ if and only if $\Psi^{-1}(v)$ is in a black (white) component in $G_A^*$ and we say it is an odd (even) vertex in $\Tilde{G}_A$ if and only if $\Psi^{-1}(v)$ is an odd (even) vertex.
Note that, somewhat artificially, in $\tilde G_{A}$ we may have two even (odd) vertices neighboring each other. In addition, we consider the connected components in $\tilde G_A$ composed of black and white vertices and we define the capacity of each such component to be the number of odd vertices in that component. See Figure~\ref{fig: map psi} for an illustration.

In what follows, all the sequences are in order according to the maximal order in $G_A$. We define $\psi_2(\mathsf{C})$ to be the sequence of capacities of black components in $\Tilde{G}_A$ and we define $\psi_3(\mathsf{C})$  to be the sequence of capacities of white components in $\Tilde{G}_A$. Furthermore, we define $\psi_4(\mathsf{C})$ to be the sequence of colors of all the branching vertices and their neighbors in $G_A$.
Now we get the mapping $\psi(\mathsf{C})=\Big(\psi_1(\mathsf{C}),\psi_2(\mathsf{C}),\psi_3(\mathsf{C}),\psi_4(\mathsf{C})\Big)$. 
\end{defi}
\begin{figure}[ht]
\centering  
\subfloat[$G_A$]{
\includegraphics[width=0.2\linewidth]{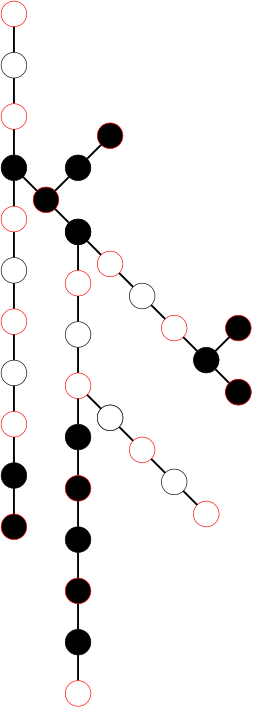}}\hspace{0.4\linewidth}
\subfloat[$G_A^*$]{
\includegraphics[width=0.25\linewidth]{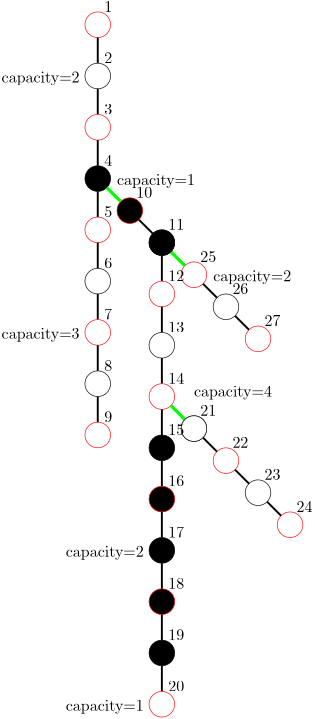}}
\caption{An illustration of ${G}_A$ and $G_A^*$. The vertices with red boundaries: odd vertices.  The vertices with black boundaries: even vertices. The numbers near the vertices are their orders in $G_A^*$. The green lines: the edges between a branching vertex and its far neighbor.}
\label{fig: cutting the leaves}
\end{figure}
\begin{figure}[ht]
\centering  
\includegraphics[width=0.9\linewidth]{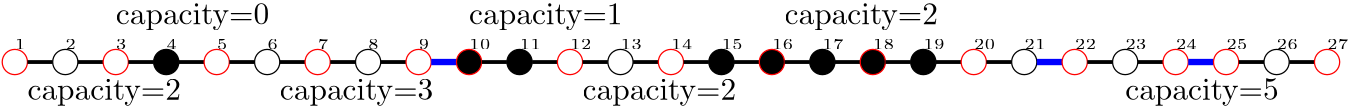}
\caption{An illustration of $\tilde{G}_A$. The blue lines: the edges that are not in $G_A^*$.}
\label{fig: map psi}
\end{figure}

The following two propositions give the properties of $\Psi$ and $\psi$. 
\begin{prop}\label{prop: Psi interpretation}
Let $A$ be an odd cluster. We construct a graph $H$ from $G^*_A$ by performing the following procedure for each branching vertex $x\in G^*_A$:
    let $p$ be the degree of $x$ and let $y_x^1,\cdots,y_x^{p-2}$ (note that $p\ge 3$) be the far neighbors of $x$; for any $1\le s\le p-2$, let $l_x^s$ be the order of $y_x^s$; and let $z_x^s$ be the vertex  with order $l_x^s-1$;  remove the edge $\{x,y_x^s\}$ and add the edge $\{y_x^s,z_x^s\}$. (See the green edges in Figure~\ref{fig: cutting the leaves} (b) and blue edges in Figure~\ref{fig: map psi} for an illustration.) Then we have $H=\tilde{G}_A$.
\end{prop}
\begin{proof}
Let $p\ge 3$.
    For each degree $p$ vertex $x\in G^*_A$, since the orders of $x$ and $y_x^s$ are not consecutive, they are not neighbors in $\tilde{G}_A$; since $y_x^s$ and $z_x^s$ have neighboring orders, they are neighbors in $\tilde{G}_A$. 
    In conclusion, if $x, y$ are neighbors in both $G^*_A$ and $\tilde G_A$ or in neither of them, then clearly the removing and adding procedure does not change the neighboring relation between $x, y$. Now it suffices to consider the pairs $x, y$ such that they are neighbors in either $G^*_A$ or $\tilde G_A$.

    First, we consider neighbors $x,y$ in $G_A^*$ that are not neighbors in $\tilde{G}_A$. Assume that $x$ has order $l$ and $l$ is smaller than the order of $y$. Since they are not neighbors in $\tilde{G}_A$, their orders are not consecutive. Since the order $1$ vertex in $G_A^*$ is a leaf, we get that $l>1$. Thus two neighbors of $x$ in $G^*_A$ (with order $l-1$ and $l+1$, respectively) are also neighbors of $x$ in $\tilde G_A$, and thus $y$ is yet another neighbor of $x$ in $G^*_A$. Thus we get that $x$ is a branching vertex in $G_A^*$ and $y$ is a far neighbor of $x$, and thus $x,y$ are not neighbors in $H$.

    Next, we consider neighbors $x,y$ in $\tilde{G}_A$ which are not neighbors in $G_A^*$. Assume that $x$ has order $l$ and $y$ has order $l+1$. Since  $x$ and $y$ are not neighbors in $G_A^*$, we get that $x$ is a leaf of $G_A^*$. Furthermore, by the definition of depth-first search, we get that $y$ is the far neighbor of a branching vertex in $G_A^*$, and thus $x, y$ are neighbors in $H$. 
    
Altogether, this completes the proof of the lemma.
    \end{proof}
\begin{prop}\label{prop: injective}
    The map $\psi$ is injective.
\end{prop}
\begin{proof}
Recall that $A$ is viewed as a fixed odd cluster (so it is known to us).
    Suppose that $\psi(C_1,\cdots,C_r)=(\mathbf{B}_1,\mathbf{B}_2,\mathbf{W},\mathbf{R})$. It suffices to recover $C_1,\cdots,C_r$ from $\mathbf{B}_1,\mathbf{B}_2,\mathbf{W},\mathbf{R}$. Note that a leaf of $G_A$ is black if and only if it has a non-zero value in $\mathbf{B}_1$. Thus we can recover the black leaves of $G_A$. Let $x_1,\cdots,x_s$ be the black leaves of $G_A$ in order and let $d_i$ be value of $x_i$ in $\mathbf{B}_1$.
    Let $G_A^0=G_A$. Recall Definition~\ref{def: psi}. We want to recover $G_A^t$ inductively. For $1\leq t\leq s$, we let $\hat{B}_{x_t}$ be the cluster containing $x_t$ with $2d_t+1$ vertices in $G_A^{t-1}$, and we next explain why such $\hat{B}_{x_t}$ is unique. Since the leaf-cutting procedure will not create new black leaves, we get that the first black leaf in $G_A^{t-1}$ is $x_t$. Since the cluster $B_{x_t}$ as in Definition \ref{def: psi} does not contain a branching vertex of $G_A^{t-1}$ and has $2d_t+1$ vertices, we get that any legitimate $\hat{B}_{x_t}$ does not contain a branching vertex and thus is unique. Furthermore, we have $\hat{B}_{x_t}=B_{x_t}$. Then we get from Definition~\ref{def: psi} that $G_A^t=G_A^{t-1}\setminus \hat{B}_{x_t}$.

  Since $\Psi$ depends only on $G_A$ and $G_A^*$, we can recover $\tilde{G}_A$ from  $G_A^*$. Furthermore, we can recover whether a vertex in $\tilde{G}_A$ is odd or not. In order to recover the white components in $G_A^*$, it suffices to recover the color of all the vertices in $\tilde{G}_A$. Note that $\tilde{G}_A$ is a line and $\mathbf{B}_2,\mathbf{W}$ are the capacities of black and white components in $\tilde{G}_A$. Since black and white components appear alternatively in $\tilde{G}_A$ and the first and last component in $\tilde{G}_A$ is white, we can recover the color of all the odd vertices in $\tilde{G}_A$ from $\mathbf B_2$ and $\mathbf W$. Since a leaf in $\tilde{G}_A$ is still a leaf in $G_A$ and all leaves in $G_A$ are white, we get that any even vertex in $\tilde{G}_A$ has two neighbors. For an even vertex $x$ in $\tilde{G}_A$ with two neighbors $y,z$, in the case that $y, z$ have the same color, we have the following: if $y, z$ are in the same colored component, then $x$ has the same color as them; if $y$ and $z$ are in different colored components, then $x$ has a different color from them. We next consider the case when $y, z$ are of different colors. Since $x$ is even, we get that $\Psi^{-1}(x)$ is not a leaf. If $\Psi^{-1}(x)$ is a branching vertex in $G_A^*$ or $N(\Psi^{-1}(x))$ contains a branching vertex in $G_A^*$, then we get its color from $\mathbf{R}$. If $\Psi^{-1}(x)$ has degree less than three in $G_A^*$ and it is not a neighbor of a branching vertex, then the neighbors of $\Psi^{-1}(x)$ in $G_A$ have neighboring orders to $\Psi^{-1}(x)$. 
  Recalling that $y,z$ are the two neighbors of $x$ in $\tilde{G}_A$, we get that they have neighboring orders to $x$ in $G_A$ and thus we see that $\Psi^{-1}(y)$ and $\Psi^{-1}(z)$ are the two neighbors of $\Psi^{-1}(x)$ in $G_A^*$. Hence we get that $\Psi^{-1}(x)$ has only one white neighbor in  $G_A^*$. Since an odd vertex is white if and only if it is in $C$ and an even vertex is white if and only if it is in $D$ (i.e., the set induced by $C$), we get that $\Psi^{-1}(x)$
  does not have two neighbors in $C$ and this means $\Psi^{-1}(x)$ is black in $G_A^*$ (since $\Psi^{-1}(x)$ cannot be induced by $C$). Therefore, we recover the colors of all vertices in $\tilde{G}_A$.
  
In conclusion, we can recover $(C_1,\cdots,C_r)$ from $(\mathbf{B}_1,\mathbf{B}_2,\mathbf{W},\mathbf{R})$ and thus the map $\psi$ is injective.
\end{proof}
Since $\psi$ is injective, it suffices to bound the enumeration of images of $\psi$.
\begin{lem}\label{lem: enumeration 1}
Let $A$ be an odd cluster with $n$ vertices. Let $r,m$ be  non-negative integers. Then we have the following.
    \begin{enumerate}
        \item $\Big|\Big\{\big(\psi_1(\mathsf{C}),\psi_2(\mathsf{C})\big): \mathsf{C} \in\tilde\sC(A,r,m)\Big\}\Big|\le\binom{n-m+r+\xi^*(G_A)}{r+\xi^*(G_A)}$.\label{item: 2}
        \item $\Big |\psi_3(\tilde\sC(A,r,m))\Big |\le\binom{m+r-1}{r-1}$.\label{item: 3}
        \item $\Big |\psi_4(\tilde\sC(A,r,m))\Big |\le2^{(d+1)\xi^*(G_A)}$.\label{item: 4}
    \end{enumerate}
\end{lem}
\begin{proof}
We first prove Item \ref{item: 3}. Note that $\psi_3(\mathsf{C})$ records the capacities of all white components in $\tilde{G}_A$. We get that the total capacity of $\psi_3(\mathsf{C})$ is the number of white odd vertices and thus is $m$. In order to bound the enumeration, it suffices to bound the number of white components in $\tilde{G}_A$. 
In light of Proposition~\ref{prop: Psi interpretation}, we wish to prove that in the removing and adding procedure therein, the number of white components will not increase. Since each component in $G_A^*$ containing a leaf is white, we see that for each branching vertex $x$ removing the edge $\{x,y_x\}$ and adding the edge $\{y_x, z_x\}$ will not increase the number of white components. 
Combined with Proposition \ref{prop: Psi interpretation}, this implies that the number of white components in $\tilde{G}_A$ is at most $r$. Then the enumeration is bounded by the number of ways to split $m$ into $r$ non-negative integers, which is $\binom{m+r-1}{r-1}$. 

Next, we prove Item \ref{item: 2}. Note that $\psi_1(\mathsf{C})$ records the capacities of all black  components containing a leaf in the leaf-cutting procedure in Definition \ref{def: psi} and $\psi_2(\mathsf{C})$ records the capacities of all black components in $\tilde{G}_A$. Thus for any odd black vertex in $G_A$, it is either recorded in $\psi_1$ or is recorded in $\psi_2$. Hence the sum of the $\ell_1$-norms of $\psi_1(\mathsf{C})$ and $\psi_2(\mathsf{C})$ is $n-m$. In order to bound the enumeration, it suffices to bound the numbers of capacities (i.e., the dimensions of these two vectors) in $\psi_1(\mathsf{C})$ and $\psi_2(\mathsf{C})$. Since $G_A$ is a tree of maximal degree $d$, we get that the number of leaves in $G_A$ is $2+\sum_{p=3}^d(p-2)\cdot|\Xi_p(G_A)|=2+\xi^*(G_A)$. Therefore, the number of capacities in $\psi_1(\mathsf{C})$ is  $2+\xi^*(G_A)$. Furthermore, the black and white components in $\tilde{G}_A$ appear alternatively and the first and last components are white. Combining with the fact that there are at most $r$ white components in $\tilde{G}_A$, we get that the number of black components in $\tilde{G}_A$ is at most $r-1$. 
In conclusion, the total dimension of $\psi_1(\mathsf{C})$ and $\psi_2(\mathsf{C})$ is at most $1+r+\xi^*(G_A)$ and thus the enumeration of $\big(\psi_1(\mathsf{C}),\psi_2(\mathsf{C})\big)$ is bounded by the number of ways to split $n-m$ into $1+r+\xi^*(G_A)$ non-negative integers, which is $\binom{n-m+r+\xi^*(G_A)}{r+\xi^*(G_A)}$. 

Finally, we prove Item \ref{item: 4}. Since there are at most $\xi^*(G_A)$ branching vertices in $G_A^*$ and each branching vertex has at most $d$ neighbors, we get that $\psi_4$ records colors for at most $(d+1)\xi^*(G_A)$ vertices. Since each vertex has two possible colors, we see that the enumeration is bounded by $2^{(d+1)\xi^*(G_A)}$.
\end{proof}
\begin{proof}[Proof of Proposition~\ref{prop: enumeration lemma step 1}]
    Combining Proposition~\ref{prop: injective} and Lemma~\ref{lem: enumeration 1} gives the desired result.
\end{proof}
\subsection{Proof of Lemma~\ref{lem: enumeration lemma final}: The general case}
In this subsection, we bound $|\sC(A,r_1,r_2,m_1,m_2)|$ in full generality. Note that although we can recover the capacities of the white components in $G_A$ from $\psi$, we cannot tell directly from $\psi$ the capacities of the white components in $G_A$. So we can only upper-bound $|\sC(A,r_1,r_2,m_1,m_2)|$ by $|\sC(A,r_1+r_2,m_1+m_2)|$, which is too weak to carry out the induction in Section~\ref{sec: single cluster}.  As a result, we need to record more so that we can directly tell the capacities of the white components in $G_A$. This motivates the following definition.
\begin{defi}\label{def: psi 4 tilde}
For two integers $1\le a<b\le |V(G_A)|$, let $G_{a,b}$ be the sub-tree of $G_A$ composed by vertices of order between $a$ and $b-1$.  

Let $J$ denote the collection of branching vertices in $G_A$. For any branching vertex $x\in J$, we say $(x,y)$ is a branching pair if $y$ is a far neighbor of $x$. Let $l_x,l_y$ be the order of $x$ and $y$ respectively. 
    
Let $\mathsf{C}=(C_1, \cdots, C_r)$ be a given sequence of disjoint odd clusters and we allow them to be adjacent. For any branching pair $(x,y)$, let $C_{x,y}$ be the white component containing $x$ in $G_{l_x,l_y}$ (we use the convention $C_{x,y}=\emptyset$ if $x$ is black).
We record the number of white components in $G_{l_x,l_y}$ with orders bigger than the order of $C_{x,y}$, and we let $\psi_5(\mathsf{C})$ be the recorded values for all the branching pairs in the dictionary order according to the depth-first search of $G_A$. In addition, we define $\tilde{\psi}(\mathsf{C})=\Big(\psi_1(\mathsf{C}),\psi_2(\mathsf{C}),\psi_3(\mathsf{C}),\psi_4(\mathsf{C}),\psi_5(\mathsf{C})\Big)$. 
\end{defi}
\begin{rmk}\label{rmk: psi 4 tilde}
    Since $\tilde{\psi}$ includes $\psi$, it is easy to get $\tilde{\psi}$ is injective by Proposition~\ref{prop: injective}.
\end{rmk}
\begin{prop}\label{prop: enumeration lemma step 2}
    Let $A$ be an odd cluster. Let $\mathsf{C}_1=(C_{1,1},\cdots,C_{1,r})$ and $\mathsf{C}_2=(C_{2,1},\cdots,C_{2,r})$ be two sequences of disjoint odd clusters in $G_A$.
    Assume that $\psi_i(\mathsf{C}_1)=\psi_i(\mathsf{C}_2)$ for $i=1,2,4,5$. If $\psi_3(\mathsf{C}_1)\neq\psi_3(\mathsf{C}_2)$, then $(|C_{1,1}|,\cdots,|C_{1,r}|)\neq(|C_{2,1}|,\cdots,|C_{2,r}|)$.
\end{prop}

\begin{proof} 
We prove this by induction on $|A|$. If $|A|=1$, then there is at most 1 cluster in $\mathsf{C}_i$ and thus the size of the cluster equals $\psi_3(\mathsf{C}_i)$. Next, we carry out the inductive proof, and we first consider a reduction as follows. Suppose that $\psi_1(\mathsf{C}_1)=\psi_1(\mathsf{C}_2)=\mathbf{B}_1$, $\psi_2(\mathsf{C}_1)=\psi_2(\mathsf{C}_2)=\mathbf{B}_2$ and $\psi_5(\mathsf{C}_1)=\psi_5(\mathsf{C}_2)=\mathbf{N}$.
We first consider the case when there exists a non-zero entry in $\mathbf{B}_1$.  Recall the leaf-cutting procedure in Definition \ref{def: psi}. Let $x$ be the first leaf with a non-zero value in $\mathbf{B}_1$.
Then we get from $\psi_1(\mathsf{C}_1)=\psi_1(\mathsf{C}_2)=\mathbf{B}_1$ that $x$ is black in both $\mathsf{C}_1$ and $\mathsf{C}_2$.  Thus $A\setminus \{x\}$ is still an odd cluster and the desired result follows from applying our induction hypothesis for $A\setminus\{x\}$, $\mathsf{C}_1$ and $\mathsf{C}_2$. Since we may apply such reduction as long as there exists a non-zero entry in $\mathbf B_1$, we may in what follows assume without loss of generality that $\mathbf{B}_1=(0,\cdots,0)$, implying that $G_A=G_A^*$.

Let $\psi_3(\mathsf{C}_1)=(a_i)_{1\le i\le r_1},\psi_3(\mathsf{C}_2)=(b_i)_{1\le i\le r_2}$ be the capacities of all white components in $\tilde{G}_A$ with respect to $\mathsf C_1$ and $\mathsf C_2$ respectively. Let $y$ be the vertex with maximal order in $N(J)$ and let $x$ be the branching vertex neighboring to $y$.
Let $s$ be the value recorded for the branching pair $(x,y)$ in $\mathbf{N}$ and let $l$ be the order of $y$ in $G_A^*$. Let $U$ be the collection of odd vertices with orders at least $l$ and let $c = |U|$. Let $z$ be the largest leaf in $G_A^*$. 
We divide our analysis into the following four cases. 

{\bf Case 1:} $a_{r_1}\neq b_{r_2}$ and at least one of them is smaller than $c$. Without loss of generality, we assume $a_{r_1}<b_{r_2}$, and thus $a_{r_1}<c$. 
Let $U_1,U_2$ be the odd cluster in $G_A$ containing the leaf $z$ with size $a_{r_1}$ and $\min\{b_{r_2},c\}$ respectively (see Figure \ref{fig: cases in injectivity} (a) and (b)). Note that both $U_1$ and $U_2$ are subsets of $U$ and thus they are well-defined. Then we get that $C_{1,r}=U_1$ and  $C_{2,r}\supset U_2$, which in turn contains $U_1$ as a proper subset.  Thus we get that $C_{1,r}\subsetneq C_{2,r}$.

{\bf Case 2:} $a_{r_1}=b_{r_2}< c$. We get that $C_{1,r}= C_{2,r}$ (similar to reasoning above). Let $A'=A\setminus C_{1,r}$, $\mathsf{C}_1'=(C_{1,1},\cdots,C_{1,r-1})$ and $\mathsf{C}_2'=(C_{2,1},\cdots,C_{2,r-1})$.
The desired result follows from applying our induction hypothesis for $A',\mathsf{C}_1'$ and $\mathsf{C}_2'$.

{\bf Case 3:} $a_{r_1}\ge c$, $b_{r_2}\ge c$ and either $x$ or $y$ is black. Since $\psi_4(\mathsf{C}_1)=\psi_4(\mathsf{C}_2)$, we get that the colors of $x$ and $y$ are the same in $\mathsf{C}_1$ and $\mathsf{C}_2$ and thus they are all black. Since $a_{r_1}\ge c$ and $b_{r_2}\ge c$, we get that $U\subset C_{i,r}$ for $i=1,2$.  Since we have that both $x$ and $y$ are black, we get that $C_{i,r}\subset U$. Therefore, $C_{i,r}= U$ for $i=1,2$ (see Figure \ref{fig: cases in injectivity} (c)), and in this case, we in fact have $a_{r_1} = b_{r_2} = c$. The desired result thus follows from applying our induction hypothesis for $(C_{1,1},\cdots,C_{1,r-1})$ and $(C_{2,1},\cdots,C_{2,r-1})$.

{\bf Case 4:} $a_{r_1}\ge c$, $b_{r_2}\ge c$ and both $x$ and $y$ are white.  Let $A'=A\setminus U$. Since $a_{r_1}\ge c$ and $b_r\ge c$, (as in Case 3) we get that $U\subset C_{i,r}$ for $i=1,2$ (see Figure \ref{fig: cases in injectivity} (d)). Our natural attempt is to apply the induction hypothesis for $A'$ and $(C_{i,1},\cdots,C_{i,r}\setminus U)$. However, an issue arises that $(C_{i,1},\cdots,C_{i,r}\setminus U)$ may not be in order because the  order of $C_{i,r}\setminus U$ may be substantially smaller than that for $C_{i, r}$. In order to address this, recall that $s$ is the value recorded for the branching pair $(x,y)$ in $\mathbf{N}$ and thus we see from Definition~\ref{def: psi 4 tilde} that there are $s$ white components with orders bigger than $C_{i,r}\setminus U$ in $G_{A'}$. Thus
the order of $C_{i,r}\setminus U$ in $G_{A'}$ should be $r-s$. As a result, it is natural to define $$\mathsf{C}_i'=(C_{i,1},\cdots,C_{i,r-s-1},C_{i,r}\setminus U, C_{i,r-s},\cdots,C_{i,r-1})$$ for $i=1,2$. For any branching vertex $z$, let $w\neq y$ be a far neighbor of $z$. Then the orders of $z$ and $w$ are smaller than $l$ (recall the maximality in the definition of $y$) and thus the value recorded for $z$ in ${\psi}_5(\mathsf{C}_i')$ equals to that in ${\psi}_5(\mathsf{C}_i)$. Thus we get ${\psi}_5(\mathsf{C}_1')={\psi}_5(\mathsf{C}_2')$. Furthermore, we have ${\psi}_1(\mathsf{C}_1')={\psi}_1(\mathsf{C}_2')=\mathbf{B}_1$, ${\psi}_2(\mathsf{C}_1')={\psi}_2(\mathsf{C}_2')=\mathbf{B}_2$ and ${\psi}_4(\mathsf{C}_1')={\psi}_4(\mathsf{C}_2')={\psi}_4(\mathsf{C}_1)|_{G_{A'}}$. Then the desired result follows from applying the induction hypothesis for $A',\mathsf{C}_1'$ and $\mathsf{C}_2'$.
\end{proof}
\begin{figure}[htbp]
    \centering
    \subfloat[Case 1]{\includegraphics[width=0.25\linewidth]{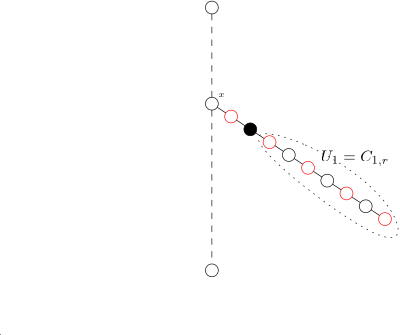}}\hspace{0.05\linewidth}
    \subfloat[Case 1]{\includegraphics[width=0.2\linewidth]{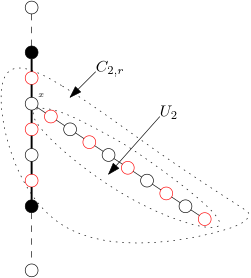}}\hspace{0.05\linewidth}
    \subfloat[Case 3]{\includegraphics[width=0.19\linewidth]{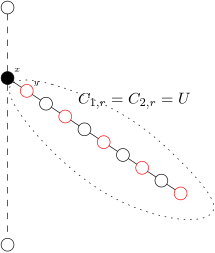}}\hspace{0.05\linewidth}
    \subfloat[Case 4]{\includegraphics[width=0.16\linewidth]{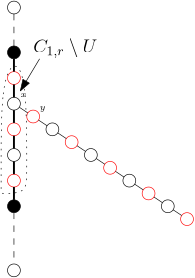}}
    \caption{An illustration of the graph $G_A$ in Proposition~\ref{prop: injective}. Dashed lines: undrawn parts of the graph $G_A$.}
    \label{fig: cases in injectivity}
\end{figure}
\begin{lem}\label{lem: enumeration 2}
Let $A$ be an odd cluster with $n$ vertices.  Let $r,m,r_1,r_2,m_1,m_2$ be  non-negative integers. Then we have 
\begin{enumerate}
    \item $\Big|\Big\{\big(\psi_1(\mathsf{C}),\psi_2(\mathsf{C})\big): \mathsf{C}\in \sC(A,r_1,m_1,r_2,m_2)\Big\}\Big|\le\binom{n-m_1-m_2+r_1+r_2+\xi^*(G_A)}{r_1+r_2+\xi^*(G_A)}.$\label{item: two type 1}
     \item $\Big|\big\{\psi_4(\mathsf{C}): \mathsf{C}\in \sC(A,r_1,m_1,r_2,m_2)\big\}\Big|\le2^{(d+1)\xi^*(G_A)}$.\label{item: two type 2}
\end{enumerate}
\end{lem}
\begin{proof}
    The proof of Item \ref{item: two type 1} is the same as the proof of Item~\ref{item: 2} in Lemma~\ref{lem: enumeration 1} and the proof of Item \ref{item: two type 2} is the same as the proof of Item~\ref{item: 4} in Lemma~\ref{lem: enumeration 1}, so we omit them here.
\end{proof}
\begin{lem}\label{lem: psi 4 and 5 enumeration bound}
    Let $A$ be an odd cluster with $n$ vertices. Let $r,m,r_1,r_2,m_1,m_2$ be  non-negative integers. Then we have 
         $$\Big|\Big\{\big(\psi_4(\mathsf{C}),\psi_5(\mathsf{C})\big): \mathsf{C}\in \sC(A,r_1,m_1,r_2,m_2)\Big\}\Big|\le 2^{(d+1)\xi^*(G_A)}\cdot\binom{r_1+r_2+4\xi^*(G_A)+1}{2\xi^*(G_A)+1}.$$
\end{lem}
We leave the proof of Lemma~\ref{lem: psi 4 and 5 enumeration bound} to the end of this section. We first show the proof of  Lemma~\ref{lem: enumeration lemma final} assuming Lemma~\ref{lem: psi 4 and 5 enumeration bound}.
\begin{proof}[Proof of Lemma~\ref{lem: enumeration lemma final} assuming Lemma~\ref{lem: psi 4 and 5 enumeration bound}]
By Proposition~\ref{prop: injective} and Remark~\ref{rmk: psi 4 tilde},
    it suffices to compute the number of tuples $(\mathbf{B}_1,\mathbf{B}_2,\mathbf{W},\mathbf{R},\mathbf{N})$ such that there exists a sequence of disjoint odd clusters $\mathsf{C}\in\sC(A,r_1,r_2,m_1,m_2)$ satisfying $\tilde{\psi}(\mathsf{C})=(\mathbf{B}_1,\mathbf{B}_2,\mathbf{W},\mathbf{R},\mathbf{N})$. Note that the number of ways to choose $(|C_1|, \ldots, |C_{r_1+r_2}|)$ is at most $\binom{m_1-1}{r_1-1}\cdot \binom{m_2-1}{r_2-1}\cdot\binom{r_1+r_2}{r_1}$. Next, we fix a choice of $(|C_1|,\cdots,|C_{r_1+r_2}|)=(a_1,\cdots,a_{r_1+r_2})$.
    By Proposition \ref{prop: enumeration lemma step 2}, we get that for any given  $\mathbf{B}_1,\mathbf{B}_2,\mathbf{R}$ and $\mathbf{N}$, there exists at most one $\mathbf{W}$ such that the following holds: there exists a sequence of disjoint odd clusters $\mathsf{C}\in\sC(A,r_1,r_2,m_1,m_2)$ satisfying $\tilde{\psi}(\mathsf{C})=(\mathbf{B}_1,\mathbf{B}_2,\mathbf{W},\mathbf{R},\mathbf{N})$ and the sequence for the cluster sizes is $(a_1,\cdots,a_{r_1+r_2})$. Thus it suffices to bound the enumeration of $\mathbf{B}_1,\mathbf{B}_2$ and $\mathbf{R},\mathbf{N}$.
    
    By  Lemma~\ref{lem: enumeration 2}, we get that the enumeration of $(\mathbf{B}_1,\mathbf{B}_2)$ is upper-bounded by $$\binom{n-m_1-m_2+r_1+r_2+\xi^*(G_A)}{r_1+r_2+\xi^*(G_A)}.$$  
    By Lemma~\ref{lem: psi 4 and 5 enumeration bound}, the enumeration of $(\mathbf{R},\mathbf{N})$ is at most $2^{(d+1)\xi^*(G_A)}\cdot\binom{r_1+r_2+4\xi^*(G_A)+1}{2\xi^*(G_A)+1}$.

     In conclusion, we get that the enumeration of   $(\mathbf{B}_1,\mathbf{B}_2,\mathbf{W},\mathbf{R},\mathbf{N})$ is upper-bounded by \begin{equation}
         \begin{aligned}\label{eq: enumeration final -1}
             &2^{(d+1)\xi^*(G_A)}\cdot\binom{m_1-1}{r_1-1}\cdot \binom{m_2-1}{r_2-1}\cdot\binom{r_1+r_2}{r_1}\\&\cdot\binom{n-m_1-m_2+r_1+r_2+\xi^*(G_A)}{r_1+r_2+\xi^*(G_A)}\cdot\binom{r_1+r_2+4\xi^*(G_A)+1}{2\xi^*(G_A)+1}.
         \end{aligned}
     \end{equation}  Combining \eqref{eq: enumeration final -1} with Lemma~\ref{lem: d=3 trifurcation relation} and the fact that $\binom{r_1+r_2+4\xi^*(G_A)+1}{2\xi^*(G_A)+1}\le 2^{r_1+r_2+4\xi^*(G_A)+1}$, it yields the desired result.
\end{proof}
Next, we prove Lemma~\ref{lem: psi 4 and 5 enumeration bound}. We start with the following lemma.
\begin{lem}\label{lem: tree order structure}
    Let $C$ be a cluster in $G_A$ with maximal order $l^+$ and minimal order $l^-$. If $x\notin C$ is a vertex with order in $(l^-,l^+)$ and $y\notin C$ is a vertex with order not in $[l^-, l^+]$, then we have $x$ and $y$ are not connected in $G_A\setminus C$.
\end{lem}
\begin{proof}
    We prove this by contradiction. Without loss of generality, we assume $y$ has order bigger than $l^+$. Assume that $x$ and $y$ are connected in $G_A\setminus C$. Let $x=x_1,\cdots,x_k=y$ be the path connecting $x$ and $y$. Then we get that $x_i\notin C$ for all $1\leq i\leq k$. Since $x$ has order smaller than $l^+$ and $y$ has order bigger than $l^+$, we get that there exists $x_i$ and $x_{i+1}$ such that $x_i$ has order smaller than $l^+$ and $x_{i+1}$ has order bigger than $l^+$. Let $l_i$ and $l_{i+1}$ be the order of $x_i$ and $x_{i+1}$ respectively. Since $x_i$ and $x_{i+1}$ are neighbors and $l_{i+1}-l_i>1$, we get that $x_i$ is a branching vertex and $x_{i+1}$ is a far neighbor of $x_i$. Consider the graph $G_A\setminus\{x_i\}$. Since $x_i\notin C$, we get that $C$ is contained in a connected component of $G_A\setminus\{x_i\}$. Let $z$ be the vertex with order $l_i-1$ and let $G_z$ be the connected component of $G_A\setminus\{x_i\}$ containing $z$. Then $G_z$ is the only connected component that contains vertices with order less than $l_i$. Note that $l^-<l_i$ and thus $C$ is contained in $G_z$. By the definition of depth-first search, we get that before searching $z$ twice, all the vertices in $G_A\setminus G_z$ have been searched twice. Thus we get that all the vertices in $G_z$ have order not in $[l_i,l_{i+1}]$. This contradicts the fact that $l_i<l^+<l_{i+1}$.
\end{proof}
\begin{defi}\label{def: hat psi 5}
Let $L$ be the collection of leaves in $G_A$ and recall that $J$ is the collection of branching vertices in $G_A$. We consider the connected components of $G_A\setminus J$. Then, each such component is a line with endpoints in $N(J)\cup L$. We say $x, y$ is a pair of \textbf{branching endpoints} if they are endpoints for such a component, for which we denote as $H_{x, y}$. If such a component has only one vertex $x$, then we use the convention that $x,x$ is a pair of branching endpoints and $H_{x,y}=\{x\}$.

    Let $\mathsf{C}=(C_1,\cdots,C_r)$ be a given sequence of disjoint odd clusters. Let $\hat{\psi}_5(\mathsf{C})_{x,y}$ denote the number of white components in $H_{x,y}$ over all pairs of branching endpoints and let $\hat{\psi}_5(\mathsf{C})$ be the sequence in the order according to the maximal order in the pairs $(x,y)$.
\end{defi}
\begin{defi}\label{def: graph H^*}
    We now construct a graph $H^*$ with vertices $J\cup N(J)\cup L$ as follows: for any $x,y\in J\cup N(J)\cup L$, an edge is added between $x,y$ in $H^*$ if and only if they are neighboring in $G_A$ or they are a pair of branching endpoints of $H_{x,y}$.  See Figure~\ref{fig: skew graph H*} (a) for an illustration.

    Recalling Definition~\ref{def: white black cluster}, we have a coloring for the vertices in $G_A$ according to $\mathsf{C}$.
    Let $H^*_{\mathsf{C}}$ be a subgraph of $H^*$ with the same vertex set and with the following edge set: for any two different vertices $x,y\in J\cup N(J)\cup L$, there is an edge between $x,y$ in $H^*_{\mathsf{C}}$ if and only if either of the following conditions holds: \begin{enumerate}
        \item  They are neighboring in $G_A$ and both of them are white.
        \item $x,y$ is a pair of branching endpoints, and all the vertices in $H_{x,y}$ are white in $G_A$.
    \end{enumerate} See Figure~\ref{fig: skew graph H*} (b) for an illustration.
\end{defi}
\begin{rmk}\label{rmk: graph H^*}
    By Definition \ref{def: graph H^*}, we get that for any two white vertices $x,y\in J\cup N(J)\cup L$, they are contained in the same white component in $G_A$ if and only if they are contained in the same white component in $H^*_{\mathsf{C}}$.
\end{rmk}
\begin{figure}[ht]
\centering  
\subfloat[$H^*$]{
\includegraphics[width=0.2\linewidth]{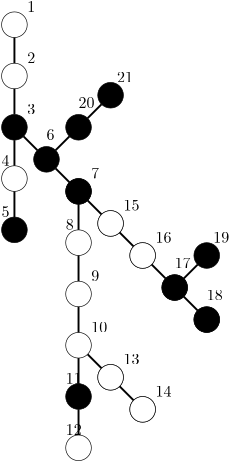}}\hspace{0.3\linewidth}
\subfloat[$H^*_{\mathsf{C}}$]{\includegraphics[width=0.2\linewidth]{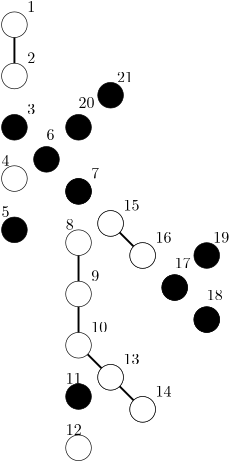}}
\caption{An illustration of $H^*$ and $H^*_{\mathsf{C}}$ where $G_A$ and $\mathsf{C}$ are given in Figure~\ref{fig: cutting the leaves}.}
\label{fig: skew graph H*}
\end{figure}
\begin{lem}\label{lem: number of white cluster by 4 +5'}
    Let $x\in J\cup N(J)$ and let $y\in N(J)$ be a far neighbor of a branching vertex.  Let $l_x$ and $l_y$ be the order of $x$ and $y$ respectively. Suppose that $l_x<l_y$. Let $\mathsf{C}=(C_1,\cdots, C_r)$ be a given sequence of disjoint odd clusters. Then we can recover the number of white components in $G_{l_x,l_y}$ from $\psi_4(\mathsf{C})$ and $\hat{\psi}_5(\mathsf{C})$. 
\end{lem}
\begin{proof}
    Let $\mathbf{BP}$ be the collection of branching pairs $(u,v)$ such that both $u$ and $v$ have orders between $l_x$ and $l_y$ and let $\mathbf{BE}$ be the collection of pairs of branching endpoints $(u,v)$ as in Definition \ref{def: hat psi 5} such that both $u$ and $v$ have orders between $l_x$ and $l_y$. Let $k$ be the number of white branching vertices in $V(G_{l_x,l_y})$ and let $E^*_{x, y}$ be the collection of branching pairs $(u,v)\in \mathbf{BP}$ such that both $u$ and $v$ are white and let $e = |E^*_{x, y}|$. Note that by Definition~\ref{def: psi}, we can get $e$ from $\psi_4(\mathsf{C})$. 
    
    We claim that the number of white components in $G_{l_x,l_y}$ is $\sum_{(u,v)\in \mathbf{BE}}\hat{\psi}_5(\mathsf{C})_{u,v}+k-e$. Suppose that  $G_{l_x,l_y}=(V_{x,y},E_{x,y})$.
    We consider the graph $G_{l_x,l_y}^\prime=(V_{x,y},E_{x,y}\setminus E_{x,y}^*)$. By Definition \ref{def: hat psi 5}, we get that the number of white components in $G_{l_x,l_y}^\prime$ is  $\sum_{(u,v)\in \mathbf{BE}}\hat{\psi}_5(\mathsf{C})_{u,v}+k$ since each white branching vertex in $G_{l_x,l_y}^\prime$ is a white component itself in addition to white components recorded in $\hat \psi_5(\mathsf C)$. Next, we will add edges in $E_{x,y}^*$ sequentially to $G_{l_x,l_y}^\prime$ (in an arbitrary order) and keep track of the change in the number of white components. For any $\{u,v\}\in E_{x,y}^*$, since both $u$ and $v$ are white, adding $\{u,v\}$ to $G_{l_x,l_y}^\prime$ connects two white components and thus decreases the number of white components by 1. In conclusion, the number of white components in $G_{l_x,l_y}$ is the number of white components in $G_{l_x,l_y}^\prime$ minus $e$ and thus is $\sum_{(u,v)\in \mathbf{BE}}\hat{\psi}_5(\mathsf{C})_{u,v}+k-e$.
\end{proof}
\begin{figure}[ht]
\centering  
\subfloat[Case 1]{
\includegraphics[width=0.18\linewidth]{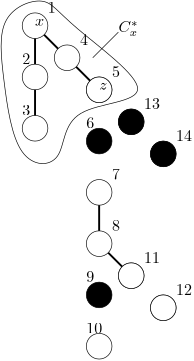}}\hspace{0.05\linewidth}
\subfloat[Case 2]{
\includegraphics[width=0.18\linewidth]{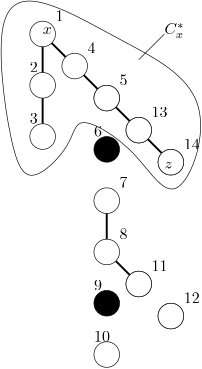}}\hspace{0.05\linewidth}
\subfloat[Case 3]{
\includegraphics[width=0.18\linewidth]{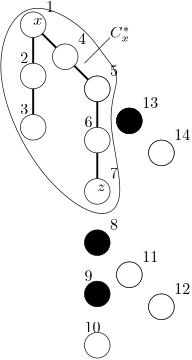}}\hspace{0.05\linewidth}
\subfloat[Case 4]{
\includegraphics[width=0.20\linewidth]{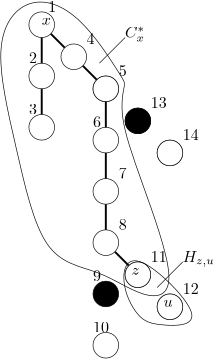}}
\caption{An illustration of the graph $H_{\mathsf{C}}^*|_{J'}$ in Lemma~\ref{lem: reduce to 4+5'}.}
\label{fig: H* on J'}
\end{figure}
\begin{figure}[ht]
\centering  
\subfloat[Case 1]{
\includegraphics[width=0.18\linewidth]{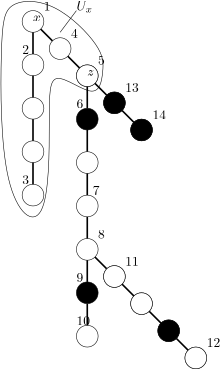}}\hspace{0.05\linewidth}
\subfloat[Case 2]{
\includegraphics[width=0.18\linewidth]{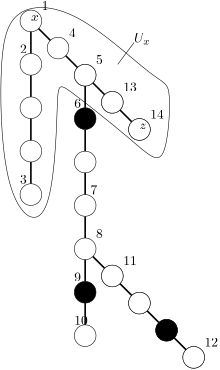}}\hspace{0.05\linewidth}
\subfloat[Case 3]{
\includegraphics[width=0.18\linewidth]{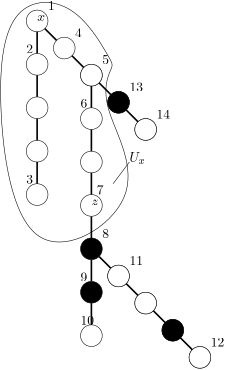}}\hspace{0.05\linewidth}
\subfloat[Case 4]{
\includegraphics[width=0.18\linewidth]{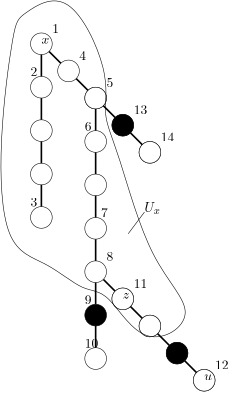}}
\caption{An illustration of the graph $G_{l_x,l_y}$ in Lemma~\ref{lem: reduce to 4+5'}.}
\label{fig: G A for H*}
\end{figure}
\begin{lem}\label{lem: reduce to 4+5'}
    Let $\mathsf{C}=(C_1,\cdots,C_r)$ be a sequence of disjoint odd clusters. We can recover $\psi_5(\mathsf{C})$ from $\psi_4(\mathsf{C})$ and $\hat{\psi}_5(\mathsf{C})$.
\end{lem}
\begin{proof}
     Fix $(x,y)$ as a branching pair and let $l_x,l_y$ be the order of $x,y$ in $G_A$ respectively.
     let $U_x$ be the white component containing $x$ in $G_{l_x,l_y}$. We want to recover the number of white components in $G_{l_x,l_y}$ with orders bigger than that of $U_x$ from $\psi_4(\mathsf{C})$ and $\hat{\psi}_5(\mathsf{C})$.  If $x$ is black then $U_x=\emptyset$, and thus $\psi_5(\mathsf{C})_{x,y}$ is the number of white components in $G_{l_x,l_y}$. By Lemma~\ref{lem: number of white cluster by 4 +5'}, we can recover $\psi_5(\mathsf{C})_{x,y}$ from $\psi_4(\mathsf{C})$ and $\hat{\psi}_5(\mathsf{C})$.
    
    If $x$ is white then we want to first determine the location of the vertex which has the maximal order in $U_x$. Recall $H_{\mathsf{C}}^*$ from Definition~\ref{def: graph H^*}. We can give $V(H_{\mathsf{C}}^*)$ an order according to their order in $G_A$. Let $J'=\big(J\cup N(J)\cup L\big)\cap V(G_{l_x,l_y})$ and let $C^*_x$ be the component containing $x$ in $H_{\mathsf{C}}^*|_{J'}$. By Remark \ref{rmk: graph H^*}, we get that $V(C^*_x)\subset V(U_x)$. Let $z$ be the vertex in $C^*_x$ with the maximal order, then $z$ has the maximal order in $J'\cap U_x$. Let $l_z$ be the order of $z$ in $G_A$.
    
    Let $w$ be the vertex in $U_x$ with the maximal order. Then we claim that either $w=z$ or there exists a vertex $u$ with order larger than $z$ such that $z,u$ is a pair of branching endpoints and $w\in H_{z,u}$. If $w\in J\cup N(J)\cup L$, since $C^*_x\subset U_x\subset V(G_{l_x,l_y})$, we get that $w\in C^*_x$ and thus $w=z$. If $w\notin J\cup N(J)\cup L$, then there exists a pair of branching endpoints $v,u$ such that $w\in H_{v,u}$ (note that potentially one may have $v = u$ due to our convention in Definition \ref{def: hat psi 5}). Without loss of generality, we assume the order of $v$ is not larger than the order of $u$. Since $w\notin J\cup N(J)\cup L$, we get that $w\neq v$ and $w\neq u$, implying that $u\neq v$ and thus the order of $v$ is smaller than the order of $u$. Note that $H_{v,u}$ is a line. We get that the order of $v$ is larger than the order of $U_x \setminus H_{v,u}$ (recalling that the order of a set is the maximal order over all vertices in the set). Combined with Remark~\ref{rmk: graph H^*} and the fact that $C_x^*\cap H_{v,u}=\{v\}$, it yields that $v$ is the vertex with maximal order in  $C^*_x$ and thus $v=z$. In conclusion, we finish the proof of the claim. 
    
    Next, we consider the following four cases.\begin{enumerate}
    \item If $z\in J$, then by the aforementioned claim and the fact that branching endpoints are contained in $\big(N(J)\cup L\big)\setminus J$, we get that $z$ is the vertex in $U_x$ with the maximal order (see Case 1 in Figures~\ref{fig: H* on J'} and \ref{fig: G A for H*}).   Applying Lemma~\ref{lem: tree order structure}, we get that for any white component $C$ in $G_{l_x,l_y}$ with order bigger than $l_z$, it is contained in $G_{l_z+1,l_y}$. Thus we get that the number of white components with orders bigger than $U_x$ equals the number of white components in $G_{l_z,l_y}$ minus 1. By Lemma~\ref{lem: number of white cluster by 4 +5'} and the assumption that $z\in J$, we get the number of white components in $G_{l_z,l_y}$ from $\psi_4(\mathsf{C})$ and $\hat{\psi}_5(\mathsf{C})$.
        \item If $z\in L$, let $H_{z,z'}$ be the connected component in $G_A\setminus J$ containing $z$ (potentially one might have $z = z'$). Then by the aforementioned claim and the fact that the order of $z$ is not smaller than the order of $z'$, $z$ is the vertex in $U_x$ with the maximal order (see Case 2 in Figures~\ref{fig: H* on J'} and \ref{fig: G A for H*}).  Applying Lemma~\ref{lem: tree order structure}, we get that for any white component $C$ in $G_{l_x,l_y}$ with order bigger than $l_z$, it is contained in $G_{l_z+1,l_y}$. Thus we get that the number of white components with orders bigger than $U_x$ equals the number of white components in $G_{l_z+1,l_y}$. Let $u$ be the vertex with order $l_z+1$. By Lemma~\ref{lem: number of white cluster by 4 +5'} and the fact that $u\in N(J)$, we get the number of white components in $G_{l_u,l_y}$ from $\psi_4(\mathsf{C})$ and $\hat{\psi}_5(\mathsf{C})$.
        \item  If $z\in N(J)$ and $z$ is the vertex with maximal order in $U_x$ (see Case 3 in Figures~\ref{fig: H* on J'} and \ref{fig: G A for H*}). Applying Lemma~\ref{lem: tree order structure}, we get that for any white component $C$ in $G_{l_x,l_y}$ with order bigger than $l_z$, it is contained in $G_{l_z,l_y}$.
        Furthermore, since the white component containing $z$ in $G_{l_z,l_y}$ is itself, we get that the white component containing $z$ in $G_{l_z,l_y}$ is the smallest white component in $G_{l_z,l_y}$.
        Thus we get that the number of white components in $G_{l_x, l_y}$ with orders bigger than that of  $U_x$ equals the number of white components in $G_{l_z,l_y}$ minus 1. By Lemma~\ref{lem: number of white cluster by 4 +5'} and the fact that $z\in N(J)$, we get the number of white components in $G_{l_z,l_y}$ from $\psi_4(\mathsf{C})$ and $\hat{\psi}_5(\mathsf{C})$.
        \item  If $z\in N(J)$ and $z$ is not the vertex with maximal order in $U_x$, then by the aforementioned claim, there exists a vertex $u$ with order larger than $z$ such that $z,u$ is a pair of branching endpoints and the vertex in $U_x$ with the maximal order is contained in $H_{z,u}$ (see Case 4 in Figures~\ref{fig: H* on J'} and \ref{fig: G A for H*}). Applying Lemma~\ref{lem: tree order structure}, we get that for any white component $C$ in $G_{l_x,l_y}$ with order bigger than $l_z$, it is contained in $G_{l_z,l_y}$.
        Furthermore, since the white component containing $z$ in $G_{l_z,l_y}$ is contained in $H_{z,u}$, we get that the white component containing $z$ in $G_{l_z,l_y}$ is the smallest white component in $G_{l_z,l_y}$.
       Thus we get that the number of white components in $G_{l_x, l_y}$ with orders bigger than that of  $U_x$ equals the number of white components in $G_{l_z,l_y}$ minus 1. By Lemma~\ref{lem: number of white cluster by 4 +5'} and the fact that $z\in N(J)$, we get the number of white components in $G_{l_z,l_y}$ from $\psi_4(\mathsf{C})$ and $\hat{\psi}_5(\mathsf{C})$.
        \end{enumerate}  
    In conclusion, in all these cases we can recover the number of white components with orders bigger than $U_x$. 
\end{proof}
Now we are ready to prove Lemma~\ref{lem: psi 4 and 5 enumeration bound}.
\begin{proof}[Proof of Lemma~\ref{lem: psi 4 and 5 enumeration bound}]
    By Lemma~\ref{lem: reduce to 4+5'}, it suffices to upper-bound 
$|\psi_4(\sC(A,r_1,m_1,r_2,m_2))|$ and $|\hat{\psi}_5(\sC(A,r_1,m_1,r_2,m_2))|$. Recalling Item \ref{item: two type 2} of Lemma~\ref{lem: enumeration 2}, we only need to bound the enumeration of $\hat{\psi}_5$. Note that there are at most $\sum_{x\in J}(\deg_{G_A}(x)-1)+1\le 2\xi^*(G_A)+1$ connected components in $G_A\setminus J$ and the number of white components in those connected components are at most $r_1+r_2+2\xi^*(G_A)$, we get that the enumeration of $\hat{\psi}_5$ is upper-bounded by $\binom{r_1+r_2+4\xi^*(G_A)+1}{2\xi^*(G_A)+1}$.
\end{proof}

\section{Proof ingredients of Theorem \ref{thm: minus separate product}}
In this section, we provide the postponed proofs for ingredients employed earlier, that is,  Lemmas \ref{lem: minimal minus subset reduction}, \ref{lem: induced cluster}, \ref{lem: minimality}, \ref{lem: leaf trifurcation relation}, \ref{lem: bound for F tri}, \ref{lem: extra points}, \ref{lem: number of extra points}, \ref{lem: multiple cluster cannot merge}, \ref{lem: C cluster to sequence reduction}, \ref{lem: dtri and cluster number relation} and \ref{lem: bound for the exponents in the multi cluster case}.
\begin{proof}[Proof of Lemma~\ref{lem: minimal minus subset reduction}]
Let $\Sigma(A,C)$ denote the collection of configurations $\sigma$ on $G$ such that there exists a unique set $B\subset N(A)$ satisfying $ \Prr(A,B)=C$ and $\sigma_B=-1,\sigma_{N(A)\setminus B}=1$. Then we get that the left-hand side of \eqref{eq: minimal minus subset reduction} equals to $\P(\sigma(T-1)\in\Sigma(A, C))$. Let $\Sigma'(A,C)$ denote the collection of configurations $\sigma$ on $G$ such that there exists a $B\subset \sB(A,C)$ satisfying $\sigma_B=-1$. Then we get that the right-hand side of \eqref{eq: minimal minus subset reduction} is at least $\P(\sigma(T-1)\in\Sigma'(A,C))$. Thus it suffices to prove that $\Sigma(A,C)\subset\Sigma'(A,C)$. 

For any configuration $\sigma\in \Sigma(A,C)$, let $B = \{y\in N(x): \sigma_y = -1\}.$ Let $B'\in\sB(A)$ be the smallest subset of $B$ such that $\Prr(A,B')=\Prr(A,B)=C$. Then we get that $B'\in \sB(A,C)$. Since $\sigma|_{B'} = -1,$ we have $\sigma\in \Sigma'(A,C)$. Therefore, we complete the proof of Lemma~\ref{lem: minimal minus subset reduction}.
\end{proof}
\begin{proof}[Proof of Lemma~\ref{lem: induced cluster}]
Let $x,y$ be two vertices in $D$. Since $C$ induces $D$, there are at least two neighbors of $x$ in $C$ and we let $x'\in C$ be one of the neighbors that maximizes the graph distance with $y$. Similarly, we let $y'\in C$ be a neighbor of $y$ that maximizes the distance with $x$. Since $C$ is connected, there exists an odd path in $C$ connecting $x'$ and $y'$, which we denote as $x' = z_1, z_2, \ldots, z_s= y'$ for some $s\geq 1$ (note that this means $z_i$ and $z_{i+2}$ has graph distance 2). Let $v_i = N(z_i) \cap N(z_{i+1})$. Then we see that $v_i\in D$ for $1\leq i\leq s-1$ and $v_1, \ldots, v_{s-1}$ is an even path. By the maximality when choosing $x'$ and $y'$, we see that $\{x, y\} = \{v_1, v_{s-1}\}$, completing the proof of the proposition. 
\end{proof}
\begin{proof}[Proof of Lemma~\ref{lem: minimality}]
We prove this by contradiction. 
   Suppose there exists $x\in C\setminus \mathsf{A}$. Let $y,z$ be two neighbors of $x$ in $B$. Let $G_y$ and $G_z$ be the  connected component of $G|_{V\setminus \{x\}}$ containing $y$ and $z$ respectively. Since $x\notin \mathsf{A}$ and $\mathsf{A}$ is an odd cluster, we get that $\mathsf{A}$ is contained in a connected component of $G|_{V\setminus \{x\}}$. Without loss of generality, we assume $G_y\cap \mathsf{A}=\emptyset$. Then we get that $\Prr(\mathsf{A},B)=\Prr(\mathsf{A},B\setminus\{y\})$ and this contradicts the minimality of $B$.
   \end{proof}
\begin{proof}[Proof of Lemma~\ref{lem: leaf trifurcation relation}]
We first show that if $C$ is an odd set and $B$ is an even set such that each vertex in $C$ has at least two neighbors in $B$, then \begin{equation}\label{eq: leaf trifurcation relation tmp}
        \xi^*(C,B)+|C|\le |B|-k
    \end{equation} where $k$ is the number of connected components in $H = G_{B\cup C}$. Without loss of generality, we assume $k=1$ because the general case follows from considering each connected component separately. 

    We apply induction to $\xi^*(C,B)$. If $\xi^*(C,B)=0$, then  we get that each vertex in $C$ has exactly two neighbors in $B$. Therefore, the graph $G_{B\cup C}$ is a line and thus $|B| = |C| + 1$, verifying \eqref{eq: leaf trifurcation relation tmp} in this base case.

    Now we suppose \eqref{eq: leaf trifurcation relation tmp} holds when $\xi^*(C,B)\le D-1$ and we consider the case for $\xi^*(C,B)= D\ge 1$. As before, we order vertices in $C\cup B$ by a depth-first search starting from a leaf of $H$. 
    Let $v$ be the first branching vertex in $C$. Let $v^-$ and $v^+$ be the neighbors of $v$ so that the orders for $v^-, v, v^+$ are increasing consecutive integers. Let $u$ be a far neighbor of $v$. Let $H^\prime$ be the graph obtained from $H$ by removing the edge $\{u, v\}$. Since $v^-$ and $v^+$ are the neighbors of $v$ in $B$ and $v$ is the only odd vertex that has different degrees in $H$ and $H^\prime$, we get that each odd vertex has at least two neighbors in $H^\prime$. Note that $H^\prime$ has two connected components. We denote the odd vertices in each component as $C_1, C_2$ and the even vertices in each component as $B_1, B_2$. Then we have \begin{equation}\label{eq: subgraph xi relation}
        \xi^*(C_1,B_1)+\xi^*(C_2,B_2)\le \xi^*(C,B)-1.
    \end{equation}  Applying the induction hypothesis to $C_1, B_1$ and $C_2, B_2$ gives \begin{align}
        &\xi^*(C_1, B_1)+|C_1|\le |B_1|-1,~~\mbox{and}\label{eq: induction hypothesis in 2.25 1}\\
        &\xi^*(C_2, B_2)+|C_2|\le |B_2|-1.\label{eq: induction hypothesis in 2.25 2}
    \end{align} The desired result thus follows from combining \eqref{eq: induction hypothesis in 2.25 1} and \eqref{eq: induction hypothesis in 2.25 2} with \eqref{eq: subgraph xi relation}.

    Now we are ready to prove \eqref{eq: leaf trifurcation relation}. Let $C=\Prr(\mathsf{A},B)$. Then each vertex $x\in \Xi_p(\mathsf{A},B)$ has $p$ neighbors in $B$ and thus $x\in C$. Hence we get that $\Xi_p(\mathsf{A},B)=\Xi_p(C,B)$ and thus $\xi^*(\mathsf{A},B)=\xi^*(C,B)$. Furthermore, vertices in different even clusters of $B$ will not be connected in $G|_{B \cup C}$. Therefore, \eqref{eq: leaf trifurcation relation tmp} implies \eqref{eq: leaf trifurcation relation}. 
\end{proof}
\begin{proof}[Proof of Lemma~\ref{lem: bound for F tri}]
    Recall \eqref{eq: def of F}. For general $r_1,r_2$, we observe that $\binom{n-m_1-m_2+r_1+r_2+p}{r_1+r_2+p}\le 2^{n-m_1-m_2+r_1+r_2+p}$ and thus we get that
\begin{align}
    F(r_1,r_2)&\le \sum_{m_1,m_2\ge 0}2^{r_1+r_2+p}\cdot \binom{m_1-1}{r_1-1}\cdot\binom{m_2-1}{r_2-1}\cdot 2^{-m_2}\cdot(1-\frac{\eps}{2})^{m_1+m_2}\nonumber\\&\le 2^{r_1+r_2+p}\cdot\sum_{m_1\ge 0}\binom{m_1-1}{r_1-1}\cdot(1-\frac{\eps}{2})^{m_1}\cdot\sum_{m_2\ge 0}\binom{m_2-1}{r_2-1}\cdot 2^{-m_2}.
    \label{eq: general F upper bound}
\end{align}
Note that for a non-negative integer $a,n$ and $0<b<1$, we have \begin{align}
    \sum^{\infty}_{m=0}\binom{m+a-1}{a-1}b^{m}&=(1-b)^{-a},\label{eq: binom equation 1}\\
    \sum^{n}_{m=0}\binom{m+a}{a}&=\binom{n+a+1}{a+1}.\label{eq: binom equation 2}
\end{align}
Combining \eqref{eq: general F upper bound} and \eqref{eq: binom equation 1}, we get that \begin{equation*}
    F(r_1,r_2)\le 2^{r_1+r_2+p}\cdot(\frac{2}{\eps})^{r_1}\cdot2^{r_2}= 2^{2r_1+2r_2+p}\eps^{-r_1}.
\end{equation*}
Next, we consider the case that $r_1=1$ and $r_2=0$. By \eqref{eq: binom equation 1}, we have
\begin{align}
     F(1,0)&= \sum_{m_1\ge 0}\binom{n-m_1+p+1}{1+p}\times2^{-n+m_1}\cdot(1-\frac{\eps}{2})^{m_1}\nonumber\\&\le \sum_{m_1\ge 0}\binom{n-m_1+p+1}{1+p}\times2^{-n+m_1}\stackrel{\eqref{eq: binom equation 1}}{\le}2^{p+2}.\nonumber
\end{align}
Finally, we consider the case that $r_1=0$ and $r_2=1$. By \eqref{eq: binom equation 2}, we have
\begin{align}
    F(0,1)&=\sum_{m_2\ge 0}\binom{n-m_2+p+1}{1+p}\times2^{-n}\cdot(1-\frac{\eps}{2})^{m_2}\nonumber\\&\le \sum_{m_2\ge 0}\binom{n-m_2+p+1}{1+p}\times2^{-n}\stackrel{\eqref{eq: binom equation 2}}{=}\binom{n+p+1}{p+2}\times 2^{-n}\le 2^{p+1}.\nonumber\qedhere
\end{align}
\end{proof}
\begin{proof}[Proof of Lemma~\ref{lem: extra points}]
    For any $x\in C\setminus A$, let $y_1,y_2$ be two neighbors of $x$ in $B$. By the minimality of $B$, we get that both $y_1$ and $y_2$ have a neighbor in $A$. Let $z_1$ and $z_2$ be the neighbor of $y_1$ and $y_2$ in $A$ respectively. Since $x\notin A$, we get that $z_1,z_2$ are in different odd clusters in $A$. In conclusion, we have that $x$ is a common odd neighbor of two odd clusters in $A$ and we complete the proof of Lemma~\ref{lem: extra points}.
\end{proof}
\begin{proof}[Proof of Lemma~\ref{lem: number of extra points}]
Construct a graph $H$ on $[k]$ where $\{i, j\}$ is an edge in 
$H$ if and only if there exists $x\in W$ such that $i, j$ are the smallest integers satisfying $\dist(A_i, x) = 2$ and $\dist(A_j, x) = 2$.
We claim that $H$ has no cycle. Otherwise, we suppose there exists a shortest distinct sequence $i_1,i_2,\cdots,i_{s}$ which forms a cycle in $H$ such that $\dist(A_{i_j},x_j)=\dist(A_{i_{j+1}},x_j)=2$ for some $x_j\in W,~ 1\le j\le s$ where we used the convention $i_{s+1}=i_1$. We consider the graph $G|_{V\setminus \{x_s\}}$, and let $G_1$ and $G_s$ be the two connected components of $G|_{V\setminus \{x_s\}}$ containing $A_{i_1}$ and $A_{i_s}$. By the choice of shortest in defining edges of $H$, we see that $x_1, \ldots, x_s$ are distinct vertices. Thus, $A_{i_j}$ and  $A_{i_{j+1}}$ are in the same connected component of $G|_{V\setminus \{x_s\}}$ for all $1\leq j\leq s-1$. This contradicts to the fact that $A_{i_1}$ and $A_{i_s}$ are in different connected components in $G|_{V\setminus \{x_s\}}$, completing the verification of our claim.
As a result, the number of edges in $H$ is at most $k-1$, implying that $|W| \leq k-1$.
\end{proof}
\begin{proof}[Proof of Lemma~\ref{lem: multiple cluster cannot merge}]
    Fix $B\in\sB_{W^*}$. Suppose that $B$ is the disjoint union of non-adjacent even clusters $B_1, \ldots, B_s$. For $1\leq i\leq s$, let $C_i$ be the odd cluster induced by $B_i$. We will first show that for any $C_i$, there exists an odd cluster $A_j^*$ such that $C_i\subset A_j^*$. By Lemma~\ref{lem: extra points}, we get that $C_i\setminus A\subset W$. Recalling Definition~\ref{def: multi cluster connect}, we get that $C_i\setminus A\subset W^*$ and thus $C_i\subset A\cup W^*$. Since $C_i$ is odd connected, we get that $C_i$ is contained in a connected component of $A\cup W^*$. Thus there exists  an odd cluster $A_j^*$ such that $C_i\subset A_j^*$.

    For any $1\le j\le k^*$, let $I_j=\{i\mid C_i\subset A_j^*\}$. Let $B^*_j=\cup_{i\in I_j}B_i$. It is easy to get that $B^*_j~(1\le j\le k^*)$ satisfies Condition 3. 
    For Condition 2, note that the even clusters $B_1, B_2, \cdots, B_s$ are not adjacent and that $B_j^*$ is the union of some even clusters $B_i$. Thus we get that $B^*_j~(1\le j\le k^*)$ satisfies Condition 2. For Condition 1, we prove using contradiction. If otherwise, there exists $B^\prime_i\subsetneq B_i^*$ such that $\Prr(A_i^*,B^\prime_i)=\Prr(A_i^*,B^*_i)$. Let $B^\prime=\cup_{j\neq i}B_j^*\cup B^\prime_i$. By Condition 2, we get that  $B_j^*$ is not adjacent to $B_i^*$ and thus $\Prr(A_i^*,B^\prime)=\Prr(A_i^*,B)$, which contradicts the fact that $B\in \sB_{W^*}$.
\end{proof}
\begin{proof}[Proof of Lemma~\ref{lem: C cluster to sequence reduction}]
Note that for any $C\in\cC$, we have $\sB(A \cup W,C\cup W^*)\subset\sB_{W^*}$. For $B\in \mathscr B_{W^*}$, let $B_1^*,\cdots,B_{k^*}^*$ be defined as in Lemma~\ref{lem: multiple cluster cannot merge} and let $B_{i,1},\cdots,B_{i,l_i}$ be the even connected components of $B_i^*$. Let $\mathsf{C}_i=(C_{i,1},\cdots,C_{i,l_i})$ be the sequence of odd clusters induced by $B_{i,1},\cdots,B_{i,l_i}$, arranged in the increasing order (recalling that the order of a set is the maximal order over all vertices in the set). Then we have $B\in \frB\Big(A\cup W,\U(\mathsf{C}_1,\cdots,\mathsf{C}_{k^*})\Big)$. 
Supposing that $\mathsf{C}_i\in \sC(A_i^*,r_{i, 1}, r_{i, 2}, m_{i, 1}, m_{i, 2})$ for each $1\leq i\leq k^*$, it suffices to prove that $\mathsf{C}_i\in \sC(A_i^*,r_{i, 1}, r_{i, 2}, m_{i, 1}, m_{i, 2},\fS_i).$

     Since $S_i(p)\subset W$, we get that for any connected component $A_{i,\#}$ of $A_i^*\setminus \cup_{p=3}^dS_i(p)$, there exists an $A_q$ such that $A_q\subset A_{i,\#}$. By the  definition of $B_i^*$, we get that for $\Prr(A_{i_1}^*,B_{i_2}^*)= \emptyset$ any $i_1\neq i_2$. Since $\Prr(A_j,B)\neq \emptyset$ for $1\le j\le k$, we get that $\Prr(A_q,B_i^*)\neq \emptyset$. 
    Letting $C_i=(\cup_{j=1}^{l_i}C_{i,j})$, we get that $C_i$ is the collection of odd vertices in $A\cup W$ induced by $B^*_i$. Therefore we get that $C_i\cap A_{q}\neq \emptyset$ and thus $C_i\cap A_{i,\#}\neq \emptyset$. In conclusion, we get that $\mathsf{C}_i\in \sC(A_i^*,r_{i, 1}, r_{i, 2}, m_{i, 1}, m_{i, 2},\fS_i)$ and this completes the proof of \eqref{eq: C cluster to sequence reduction 1}.
    
In order to prove \eqref{eq: C cluster to sequence reduction 2}, it suffices to prove $B^*_i\in \frB(A_i^*,\mathsf{C}_i)$. It is clear that $B_i^*$ is a disjoint union of $B_{i,1},\cdots,B_{i,l_i}$ and $C_{i,j}$ is induced by $B_{i,j}$. Furthermore, by Lemma~\ref{lem: multiple cluster cannot merge}, we get that $B^*_i\in\mathscr{B}(A_i^*)$. Thus, in conclusion, we get that $B^*_i\in \frB(A_i^*,\mathsf{C}_i)$.
\end{proof}
\begin{proof}[Proof of Lemma~\ref{lem: dtri and cluster number relation}]
Let $\mathsf{C}\in \sC(\mathsf{A},r_1,r_2,m_{1},m_{2},\fS)$ and $B\in\sB(\mathsf{A},\mathsf{C})$. Let $C$ be the collection of odd vertices induced by $B$ and let $\hat{G}$ be the graph restricted on $B\cup C$. Then $r_1+r_2$ is the number of connected components in $\hat{G}$. Since $C$ is the union of clusters in $\mathsf{C}$, we have that $\cup_{p=3}^dS(p)\subset C$.  For each $x\in S(p)$, we perform the following procedure: Let $y_1,\cdots,y_p$ be the neighbors of $x$ such that $N(y_i)\cap \mathsf{A}\setminus\{x\}\neq \emptyset$. If $y_i\in B$, then we do nothing; if $y_i\notin B$, then we remove the edge of $\{x,y_i\}$ from $G$. We denote by $G^{\#}$ the final graph we obtain at the end of our procedure. Since for any removed edge $\{x,y\}$, we have $y\notin B$ and thus $\{x,y\}\notin E(\hat{G})$. Let $\hat{G}^{\#}=G^{\#}|_{B\cup C},$ then we get that $E(\hat{G})\subset E(\hat{G}^{\#})$. Furthermore, since $S(p)\subset C$, we get that $\deg_{\hat{G}}(x)\ge 2$ and thus $x$ is not isolated in $G^{\#}.$ Note that for any $x\in S(p)$, the number of removed edges from $x$ is $p-\deg_{\hat{G}}(x)$. Thus we get that the number of edges we have removed is 
$$\sum_{p\ge 3}\sum_{x\in S(p)}\big(p-\deg_{\hat{G}}(x)\big)=\sum_{p\ge 3}\sum_{x\in S(p)}\big((p-2)-(\deg_{\hat{G}}(x)-2)\big)\ge \kappa(\fS)-\xi^*(\mathsf{A},B).$$ 
Since $G$ is a tree, we get that the final graph $G^{\#}$ has $q\ge \kappa(\fS)-\xi^*(\mathsf{A},B)+1$ connected components, denoted as $G^{\#}_1,\cdots,G^{\#}_q$. Let $S=\cup
_{p=3}^dS(p)$ and let $G_1,\cdots,G_s$ be the connected components of $G|_{V(G)\setminus S}$. Then by the construction of $G^{\#}$, we get that for any connected component $G^{\#}_i$ of $G^{\#}$, there exists $1\leq j\leq s$ such that $V(G_j)\subset V(G^{\#}_i)$. Recalling Definition~\ref{def: multi cluster connect 1}, we have $C\cap V(G_j)\neq \emptyset$ and thus $C\cap V(G_i^{\#})\neq \emptyset$ for any $1\le i\le q$.  Therefore, we get that the number of connected components of $\hat{G}^{\#}$ is at least $q$. Combining with the fact that $V(\hat{G})=V(\hat{G}^{\#})$ and $E(\hat{G})\subset E(\hat{G}^{\#})$, we get that the number of connected components of $\hat{G}$ is at least $q$. Recalling that the number of connected components of $\hat{G}$ is $r_1+r_2$,
we get that $r_1+r_2\ge q\ge \kappa(\fS)-\xi^*(\mathsf{A},B)+1$.
\end{proof}
\begin{proof}[Proof of Lemma~\ref{lem: bound for the exponents in the multi cluster case}]
We first prove \eqref{eq: cluster number edge relation}. We will add vertices in $W^*$ sequentially to $A$ (in an arbitrary order) and keep track of the change in the number of connected components. Recall the definition of $S_i(p)$, we get that for each $x \in W^*$ which is a type $p$ double trifurcation in $A$, adding $x$ to $A$ will connect $p$ odd connected components and thus decrease the number of connected components by $p-1$. Furthermore, if $x\in W^*$ is not a double trifurcation in $A$, adding $x$ to $A$ will only connect two odd connected components and thus decrease the number of connected components by 1. Combining these cases together completes the proof of \eqref{eq: cluster number edge relation}. Furthermore, we also get that \begin{equation}\label{eq: exponent calculation tmp}
    |W^*|\ge \frac{1}{d-1}\cdot \sum_{i=1}^{k^*}(\kappa(\fS_i)+1).
\end{equation}

We now prove \eqref{eq: cluster merge trifurcation relation}. For any sets $D_1\subset D_2$, we have that if a vertex $x$ is a  single trifurcation of $D_2$ but not a single trifurcation of $D_1$, then we have $x\in \cup_{u\in D_2\setminus D_1}N(u)$. Therefore we get that $\Tri(D_2)-\Tri(D_1)\le d|D_2\setminus D_1|$. Since  $A^*_1, \ldots, A^*_{k^*}$ are  non-adjacent clusters, we have that $\Tri(\cup_{1\leq i\leq k^*} A^*_i) = \sum_{1\leq i\leq k^*} \Tri(A^*_i)$. Applying the preceding inequality with $D_2 = \cup_{1\leq i\leq k^*} A^*_i$ and $D_1 = A$, we derive the first inequality in \eqref{eq: cluster merge trifurcation relation} by noting that $|\cup_{1\leq i\leq k^*} A^*_i\setminus A| \leq  |W| \leq k$ (recall Lemma \ref{lem: number of extra points}).

Similar to single trifurcations, we can also get the following inequality for double trifurcations: $\Dtri(D_2)-\Dtri(D_1)\le d(d-1)|D_2\setminus D_1|$. By monotonicity, it suffices to prove the second inequality of \eqref{eq: cluster merge trifurcation relation} assuming $W = W^*$, which implies that $A^*_1, \cdots, A^*_{k^*}$  have pairwise distance strictly larger than $2$. Thus, $\Dtri(\cup_{1\leq i\leq k^*} A^*_{i}) = \sum_{1\leq i\leq k^*} \Dtri(A^*_{i})$. Applying the aforementioned inequality with $D_2 = \cup_{1\leq i\leq k^*} A^*_{i}$ and $D_1 = A$, we then get that 
\begin{equation}\label{eq: cluster merge trifurcation relation tmp}
    \sum_{i=1}^{k^*}\Dtri(A_i^*)\le d(d-1)k+\Dtri(A).
\end{equation} In addition, if $x$ is a double trifurcation of $A$ but is not a double trifurcation of $A_i$ for any $1\leq i\leq k$, then there exist $A_i$ and $A_j$ such that $\dist(x,A_i)=\dist(x,A_j)=2$. Hence we get that $x\in W$. In conclusion we get that $\Dtri(A)\le\sum_{i=1}^k \Dtri(A_i)+|W|$. Combined with \eqref{eq: cluster merge trifurcation relation tmp} and Lemma~\ref{lem: number of extra points}, this yields the second inequality of \eqref{eq: cluster merge trifurcation relation}. 

Finally, we prove \eqref{eq: dtri in extra point number bound}.
Combining \eqref{eq: exponent calculation tmp} with the second equality in \eqref{eq: cluster number edge relation}, we get \eqref{eq: dtri in extra point number bound} and thus complete the proof of the lemma.
\end{proof}

\section*{Acknowledgements}
J. Ding is supported by NSFC Tianyuan Key Program Project No.12226001, NSFC Key Program Project No.12231002, and by New Cornerstone Science Foundation through the XPLORER PRIZE. F. Huang is supported by the Beijing Natural Science Foundation Undergraduate Initiating Research Program (Grant No. QY23006).
We warmly thank Tom Spencer for introducing the problem to us.

\small
\bibliography{myref}
\bibliographystyle{abbrv}
\end{document}